\documentclass[letterpaper,11pt]{amsart}
\usepackage[margin=1in]{geometry}

\usepackage{enumerate,tikz,amsmath,hyperref}

\theoremstyle{plain}
\newtheorem{theorem}{Theorem}[section]
\newtheorem*{theorem1.1}{Theorem 1.1}
\newtheorem*{theorem1.2}{Theorem 1.2}
\newtheorem{lemma}[theorem]{Lemma}
\newtheorem{corollary}[theorem]{Corollary}
\newtheorem{proposition}[theorem]{Proposition}
\newtheorem{example}[theorem]{Example}

\theoremstyle{definition}
\newtheorem*{definition*}{Definition}
\newtheorem{definition}[theorem]{Definition}
\theoremstyle{remark}
\newtheorem{remark}[theorem]{Remark}

\numberwithin{equation}{section}

\newcommand\remove[1]{}

\def\M{\mathcal{M}}
\def\N{{\mathbb{N}}}
\def\R{\mathbb{R}}
\def\Z{\mathbb{Z}}

\def\Pr{\mathbb{P}}

\def\f2{\mathbb{F}_2}
\def\one{\mathsf{1}}

\def\sbs{\subseteq}
\def\eps{\varepsilon}
\def\os{\oslash}

\def\V{\hskip0.01cm{\rm V}}
\def\dist{\hskip0.02cm{\rm dist}\hskip0.01cm}

\def\supp{\hskip0.02cm{\rm supp}\hskip0.01cm}
\newcommand{\tc}{{\rm TC}\hskip0.02cm}

\newcommand{\ep}{\varepsilon}

\newcommand{\diam}{{\rm diam}\hskip0.02cm}

\DeclareMathOperator{\EMD}{EMD}
\DeclareMathOperator{\TC}{TC}
\def \intr {\mathrm{int}}

\def \E {\mathcal{E}}
\DeclareMathOperator{\Per}{Per}
\def \sbs {\subseteq}

\newcommand{\cg}[1]{{\color{blue}#1}}
\newcommand{\cgtwo}[1]{{\color{orange}#1}}

\begin{document}

\title{Lower Estimates for $L_1$-Distortion of Transportation Cost Spaces}

\author[Chris Gartland]{Chris Gartland$^1$}
\address{Department of Mathematics and Statistics, University of North Carolina at Charlotte, University City Blvd, Charlotte, NC 28223, U.S.A.}

\author{Mikhail Ostrovskii}
\address{Department of Mathematics, St. John's University, New York, NY, USA}

\thanks{$^1$The first named author was supported by the National Science Foundation under Grant Number DMS-2546184}

\begin{abstract} Quantifying the degree of dissimilarity between two probability distributions on a finite metric space $X$ is a fundamental task in Computer Science and Computer Vision. A natural dissimilarity measurement based on optimal transport is the earth mover's distance (EMD), also known as the Kantorovich metric or Wasserstein-1 metric. We denote the metric space of probability measures on $X$ equipped with the earth mover's distance as $\EMD(X)$, called the earth mover's space. A key technique for analyzing this metric -- pioneered by Charikar \cite{Cha02} and Indyk and Thaper \cite{IT03} -- involves constructing low-distortion embeddings of $\EMD(X)$ into the Lebesgue space $L_1$. The best upper bound for the distortion of an embedding of $\EMD(X)$ with $|X|=n$ into $L_1$ is $O(\log n)$. This result follows from a combination of Charikar’s work (which builds on \cite{KT02}) and the seminal result by Fakcharoenphol, Rao, and Talwar \cite{FRT04}. Moreover, it is well known that expander graphs yield a matching lower bound of $\Omega(\log n)$ for $L_1$-distortion, showing that the upper bound can be tight.

It became a key problem to investigate whether the upper bound of $O(\log n)$ can be improved for important classes of metric spaces known to admit low-distortion embeddings into $L_1$. In the context of Computer Vision, grid graphs — especially planar grids — are among the most fundamental. Indyk posed in \cite[Problem 2.17]{MN11} the related problem of estimating the $L_1$-distortion of the space of uniform distributions on $n$-point subsets of $\R^2$. The Progress Report in \cite{MN11}, last updated in August 2011, highlighted two key results: first, the work of Khot and Naor \cite{KN06} on Hamming cubes, which showed that the $L_1$-distortion of $\EMD(\{0,1\}^n)$ is of the order $n$, and second, the result of Naor and Schechtman \cite{NS07} for planar grids, which established that the $L_1$-distortion of $\EMD(\{0,\dots,n\}^2)$ is $\Omega(\sqrt{\log n})$.

Our first result is the improvement of the lower bound on the $L_1$-distortion of $\EMD(\{0,\dots,n\}^2)$ to $\Omega\left(\log n\right)$, matching the universal upper bound up to multiplicative constants. The key ingredient allowing us to obtain these sharp estimates is a new Sobolev-type inequality for scalar-valued functions on the grid graphs. Our method is also applicable to many recursive families of graphs, such as diamond and Laakso graphs. We obtain the sharp distortion estimates of $\log n$ in these cases as well.
\end{abstract}

\keywords{diamond graphs, earth mover's distance, grid graphs, metric embeddings, slash product, Sobolev inequality, Wasserstein metric}

\subjclass[2020]{51F30 (05C76, 30L05, 46B03, 49Q22, 68R12, 68W25)}

\maketitle
\vspace{-.5in}
\tableofcontents

\section{Introduction} \label{sec:intro}
Let $(X,d)$ be a finite metric space and $\M_0(X)$ the set of signed measures on $X$ with 0 total mass, meaning $\mu(X) = 0$ for $\mu \in \M_0(X)$. A natural and important interpretation of such a measure is considering it as a {\it transportation problem}: one needs to transport a certain product from locations $x \in X$ where $\mu(\{x\})>0$ to locations $y \in X$ where $\mu(\{y\})<0$. One can easily see that a transportation problem $\mu$ can be represented as
\begin{equation}\label{E:TranspPlan} \mu=a_1(\delta_{x_1}-\delta_{y_1})+a_2(\delta_{x_2}-\delta_{y_2})+\dots+ a_n(\delta_{x_n}-\delta_{y_n}),\end{equation}
where $a_i \ge 0$, $x_i,y_i\in X$, and $\delta_x$ is the unit point mass measure at $x \in X$. We call each such representation a {\it transportation plan} for $\mu$; it can be interpreted as a plan of moving $a_i$ units of the product from $x_i$ to $y_i$. The {\it cost} of the transportation plan \eqref{E:TranspPlan} is defined as $\sum_{i=1}^n a_id(x_i,y_i)$. The {\it transportation cost norm} $\|\mu\|_{\tc}$ on the linear space $\M_0(X)$ is the infimum of costs of transportation plans for $\mu$ satisfying \eqref{E:TranspPlan}. The corresponding normed space $(\M_0(X),\|\cdot\|_\tc)$ is called the {\it transportation cost space} of $X$, denoted $\tc(X)$. (Other names commonly used in the literature are {\it  Lipschitz-free space, Kantorovich-Rubinstein space}, or {\it Arens-Eells space}). Kantorovich duality \cite[Particular Case~5.16]{Vil09} gives a dual formulation of the norm  $\|\mu\|_{\tc} = \sup\{|\int_X fd\mu|: f: X \to \R \text{ is 1-Lipschitz}\}$. The theory of transportation cost spaces, introduced by Kantorovich and Gavurin \cite{Kan42, KG49}, was initially developed as a study of special norms related to function spaces on finite metric spaces.

The {\it earth mover's distance} between probability measures $\mu$ and $\nu$ on $X$ is the quantity $\EMD(\mu,\nu) = \|\mu-\nu\|_{\tc}$. We denote the space of probability measures equipped with this metric by $\EMD(X)$. The term has its roots in the work of Monge (1781). It was introduced to Computer Science by Rubner, Tomasi, and Guibas \cite{RTG98,RTG00}, who used the name Earth Mover's Distance to describe this metric in Computer Vision. The importance of $\EMD$ in Computer Vision owes to the fact that when images are represented as probability distributions, the earth mover's distance provides a natural notion of dissimilarity. In this context, the most relevant underlying metric spaces $X$ are the 2-dimensional planar grids $\{0,\dots n\}^2$, as distributions on them model 2-dimensional images.

Charikar \cite{Cha02} and Indyk-Thaper \cite{IT03} noticed that one of the useful ways for studying $\EMD(X)$ on finite metric spaces is by using a low-distortion embedding of $\EMD(X)$ into $L_1$.
 
Recall that the {\it $L_p$-distortion} of a metric space $(X,d)$ is the infimal $D \geq 1$ such that there exists a measure space $(\Omega,\mathcal{A},M)$ and a map $f: X \to L_p(M)$ with distortion $D$, meaning $d(x,y) \leq \|f(x)-f(y)\|_p \leq Dd(x,y)$ for every $x,y\in X$. This quantity is denoted by $c_p(X)$. 

An important starting point for the Charikar-Indyk-Thaper results: There exists an absolute constant $C<\infty$ such that for any $n$-element metric space $X$, the inequality
\begin{equation} \label{eq:O(log(n))}
    c_1(\EMD(X))\le C\log n
\end{equation}
holds. This result was proved in two steps:
 
(1) Charikar \cite{Cha02} (implicitly) proved that $c_1(\EMD(X))$ is bounded by the stochastic distortion of $X$ into dominating tree metrics. This was observed by Indyk-Thaper \cite[page~3]{IT03}. The theory of stochastic approximation by dominating trees was initiated by Bartal \cite{Bar96}.
 
(2) Fakcharoenphol-Rao-Talwar \cite[Theorem~2]{FRT04}: There is an absolute constant $C<\infty$ such that the stochastic dominating tree distortion is $\le C\log n$ for each $n$-element metric space $X$.

It became very important to determine whether the universal estimate $O(\log n)$ can be improved for standard families of finite metric spaces. It was immediately obvious that it is not improvable for spaces with the maximal order of $L_1$-distortions, for example expander graphs. Thus, the main focus is on families with uniformly bounded $L_1$-distortions, especially ones most relevant to Computer Vision, like the planar grids $\{0,\dots n\}^2$.

The main achievements in this direction since then were:

(1) A tight lower bound for the Hamming cubes $c_1(\EMD(\{0,1\}^n))= \Omega(n) = \Omega(\log |\{0,1\}^n|)$ was obtained by Khot-Naor in \cite[Corollary~2 on p.~831]{KN06}.\footnote{Here and throughout, we write $|A|$ to denote the cardinality of $A$ whenever $A$ is a set.}

(2) For planar grids, the lower estimate $c_1(\EMD(\{0,\dots,n\}^2))= \Omega(\sqrt{\log n})$ was obtained by Naor-Schechtman \cite[Theorem~1.1]{NS07}.

(3) Baudier-Gartland-Schlumprecht \cite[Theorem~A]{BGS23} obtained a similar estimate \\$c_1(\EMD(D_n)) = \Omega(\sqrt{n}) = \Omega(\sqrt{\log|V(D_n)|})$ for the diamond graphs $D_n=(V(D_n),E(D_n))$.

Our goal is to advance significantly the results (2) and (3). See Theorems \ref{thm:c1(TC(grids))} and \ref{thm:c1(TC(slash))} below. Namely, (i) we achieve the optimal order of distortion, (ii) our proofs use less technical machinery than the results in \cite{NS07} and \cite{BGS23}, and (iii) we prove the estimate for many more recursive families of graphs, in addition to diamond graphs.

{\bf Reduction to linear maps:}
An important part of the Naor-Schechtman proof is a reduction to linear maps into finite-dimensional $\ell_1^m$ spaces \cite[Lemma~3.1]{NS07}. That is, if we let $c_{1,{\rm lin}}(\tc(X))$ denote the infimal distortion among all $m\in\N$ and all {\it linear} maps $f: \tc(X) \to \ell_1^m$ whenever $X$ is a finite metric space, then
\begin{equation}\label{E:NSred}
    c_1(\EMD(X)) = c_1(\tc(X)) = c_{1,{\rm lin}}(\tc(X)).
\end{equation}
Naor-Schechtman stated this reduction for $X$ being a planar grid, but their proof in fact works for any finite metric space. We present the proof in \S\ref{sec:NSred}. 
 We will take advantage of this reduction and henceforth work with $\tc(X)$ in place of $\EMD(X)$.

Note that combining \eqref{eq:O(log(n))} and \eqref{E:NSred} we get
\begin{equation}\label{E:UpperFRT}
    c_1(\tc(X))=O(\log|X|).
\end{equation}

\begin{figure}
\centering
\begin{minipage}{0.4\textwidth}
\centering
\begin{tikzpicture}[scale=1.8, every node/.style={circle,draw,fill=white,inner sep=1pt}]
  \node[minimum size=8pt] (s) at (-1,0) {\color{blue}\small$s_D$};
  \node[minimum size=8pt] (t) at (1,0) {\color{blue}\small$t_D$};
  \node[minimum size=8pt] (0+) at (0,1) {};
  \node[minimum size=8pt] (0-) at (0,-1) {};

  \node[draw=none,fill=none] at (0,0) {~\remove{\large$\emptyset$}};  
  
  \draw (s) -- (0+) -- (t) -- (0-) -- (s);
\end{tikzpicture}
\end{minipage}%
\begin{minipage}{0.4\textwidth}
\centering

\begin{tikzpicture}[scale=1.8, every node/.style={circle,draw,fill=white,inner sep=1pt}]
  \node[minimum size=8pt] (s) at (-1,0) {\color{blue}\small$s_D$};
  \node[minimum size=8pt] (t) at (1,0) {\color{blue}\small$t_D$};
  \node[minimum size=8pt] (0+) at (0,1) {};
  \node[minimum size=8pt] (0-) at (0,-1) {};
  
  \node[minimum size=6pt] (1+) at (-1/2-1/4,1/2+1/4){};
  \node[draw=none,fill=none] at (-.5,.5) {~\remove{1}};
  \node[minimum size=6pt] (1-) at (-1/2+1/4,1/2-1/4){};
  
  \node[minimum size=6pt] (2+) at (1/2+1/4,1/2+1/4){};
  \node[draw=none,fill=none] at (.5,.5) {~\remove{2}};
  \node[minimum size=6pt] (2-) at (1/2-1/4,1/2-1/4){};
  
  \node[minimum size=6pt] (4+) at (1/2+1/4,-1/2-1/4){};
  \node[draw=none,fill=none] at (.5,-.5) {~\remove{4}};
  \node[minimum size=6pt] (4-) at (1/2-1/4,-1/2+1/4){};
  
  \node[minimum size=6pt] (3+) at (-1/2-1/4,-1/2-1/4){};
  \node[draw=none,fill=none] at (-.5,-.5) {~\remove{3}};
  \node[minimum size=6pt] (3-) at (-1/2+1/4,-1/2+1/4){};

  \draw (s) -- (1+) -- (0+) -- (1-) -- (s);
  \draw (0+) -- (2+) -- (t) -- (2-) -- (0+);
  \draw (t) -- (4+) -- (0-) -- (4-) -- (t);
  \draw (0-) -- (3+) -- (s) -- (3-) -- (0-);
\end{tikzpicture}
\end{minipage}%
\caption{The first two diamond graphs}
\label{fig:diamonds}
\end{figure}
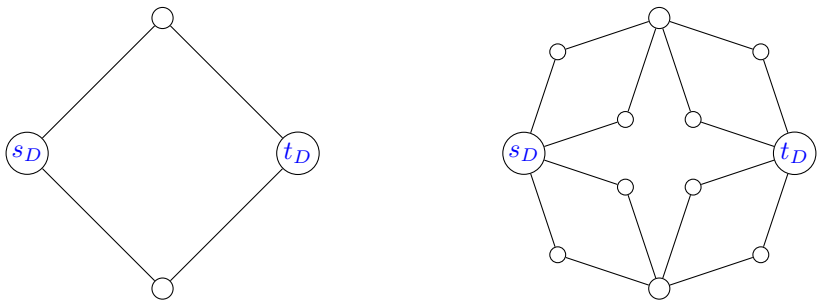

{\bf On $\os$ powers (slash powers):} The {\it diamond graphs} form a recursive family of graphs that play a significant role in various scientific areas, including Physics \cite{BO79,KG81}, Analysis \cite{JS09,OR17}, and Computer Science \cite{IW91, GNRS04, NR03, BC05, LN04}. See Figure~\ref{fig:diamonds}.
They fall into a class of graphs that we refer to as {\it $\os$ powers} (using notation originating in \cite{LR10}). The {\it $\os$ product} is a binary operation on graphs, accepting as input two $s$-$t$ graphs and returning another $s$-$t$ graph, where an {\it $s$-$t$} graph is a directed graph $H$ equipped with two distinguished vertices called $s_H$ and $t_H$. Specifically, for two $s$-$t$ graphs $G$ and $H$, the graph $G \os H$ is obtained by replacing every edge $e$ of $G$ with a copy of $H$, where the initial vertex of $e$ is replaced by $s_H$, and the terminal vertex of $e$ is replaced by $t_H$. If $G$ is an $s$-$t$ graph, the $n$th $\os$ power of $G$ is defined by the recursive formula $G^{\os 1} := G$ and $G^{\os n} = G^{\os (n-1)} \os G$. In this article, we will only consider $\os$ powers in the restricted setting where every edge in the base graph $G$ belongs to a directed geodesic from $s$ to $t$, and $G$ contains a cycle. See Definitions \ref{D:EdgeRep} and \ref{D:SlashPr} for detailed definitions of the notions we use. The most famous sequences of $\os$ powers are the diamond graphs mentioned above and the {\it Laakso graphs} $\{La_n\}_{n=1}^\infty = \{La_1^{\os n}\}_{n=1}^\infty$, first defined by Lang-Plaut \cite[Figure~2]{LP01} who were inspired by an earlier construction of Laakso \cite{Laa02}. See Figure~\ref{F:Laakso} for a picture of the first Laakso graph $La_1$. Lang and Plaut's construction of Laakso graphs was motivated by an embedding problem in Analysis, and they currently play a significant role both in that field and in Computer Science \cite{LN04, JS09, CK13, NPS20, DKO21, NY22}.

\begin{figure}
\begin{center}
{
\begin{tikzpicture}
  [scale=.25,auto=left,every node/.style={circle,draw}]
  \node[minimum size=8pt] (s) at (0,16) {\color{blue}$s$};
  \node[minimum size=8pt] (t) at (32,16) {\color{blue}$t$};
  
  \node (n2) at (8,16)  {\hbox{~~~}};
  \node (n3) at (16,12) {\hbox{~~~}};
  \node (n4) at (16,20)  {\hbox{~~~}};
  \node (n5) at (24,16)  {\hbox{~~~}};

\foreach \from/\to in {s/n2,n2/n3,n2/n4,n4/n5,n3/n5,n5/t}
    \draw (\from) -- (\to);

\end{tikzpicture}
} \caption{The first Laakso graph $La_1$.}\label{F:Laakso}
\end{center}
\end{figure}
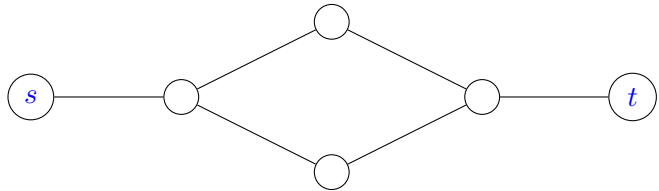

Dilworth, Kutzarova, and Ostrovskii \cite{DKO20} initiated the study of transportation cost spaces on $\os$ powers. For a large class of such families $\{G_n\}_{n=1}^\infty$, they proved lower estimates of the form $d_{\rm BM}(\tc(G_n),\ell_1^{m(n)})=\Omega\left(\frac{\log m(n)}{\log\log m(n)}\right)$ on the Banach-Mazur distance. Here, $m(n) = \dim(\tc(G_n))$, and the {\it Banach-Mazur distance} $d_{\rm BM}(V,W)$ between two normed spaces $V,W$ of the same finite dimension is the infimal distortion among all linear isomorphisms between $V$ and $W$. For diamond graphs $\{D_n\}_{n=1}^\infty$, the estimate of \cite{DKO20} was exact up to a constant,
$d_{\rm BM}(\tc(D_n),\ell_1^{m(n)})=\Theta(\log m(n))$. In the same article, Dilworth, Kutzarova, and Ostrovskii had proposed the problem of finding lower bounds on $c_1(\tc(D_n))$ \cite[Problem~6.6]{DKO20}. Since then, the only result on the problem is the previously mentioned work \cite{BGS23} yielding  $c_1(\tc(D_n))=\Omega(\sqrt{\log|V(D_n)|})$.

\subsection{Statement of Main Results}
Our first main result is an optimal improvement on Naor-Schecht\-man's lower bound for the planar grids, as we show that the $L_1$-distortion of $\tc(\{0,\dots n\}^2)$ matches the universal upper bound.

\begin{theorem} \label{thm:c1(TC(grids))}
$c_1(\tc(\{0,\dots,n\}^2))= \Theta(\log n)$.
\end{theorem}

Let $M_d^n$ denote the subset of $\EMD(\R^d)$ consisting of uniform distributions on $n$-point subsets of $\R^d$. In \cite[Problem~2.17]{MN11}, Indyk posed the problem of finding the asymptotics of $c_1(M_2^n)$. The articles \cite{AIK08,BI14} established results related to upper bounds on $c_1(M_d^n)$, and the algorithmic problem of approximating distances in $M_d^n$ was considered in \cite{CJLW22, JWZ24}. It is plausible that Theorem~\ref{thm:c1(TC(grids))}, or the methods used in its proof, could be useful in attacking Indyk's problem by providing lower bounds on the distortion. Indeed, by using Bourgain's discretization theorem, Giladi, Naor, and Schechtman \cite[Corollary~1.5]{GNS12} were able to directly apply Naor-Schechtman's lower bound $c_1(\EMD(\{0,\dots n\}^2)) = \Omega(\sqrt{\log n})$ to obtain nontrivial lower bounds on $c_1(M_2^n)$ (see also \cite[\S4]{NS07}). In Theorem \ref{T:EMDuniformR2}, we use results of \cite{NS07} and \cite{GNS12} to prove $c_1(M_2^s)=\Omega(\log s)$.

Our second main result is an analogue of Theorem~\ref{thm:c1(TC(grids))} for a large class of $\os$ powers.

\begin{theorem} \label{thm:c1(TC(slash))}
Let $G$ be any $st$-graph that is not an $st$-path and has at least three vertices. Then there exists a constant $C<\infty$ (depending on $G$ but not $n$), such that $c_1(\TC(G^{\os n})) \geq C^{-1}\log |V(G^{\os n})|$ for every $n\geq 0$.    
\end{theorem}

Theorem~\ref{thm:c1(TC(slash))} shows that, like the planar grids, each sequence in this class of graphs exhibits the largest possible $L_1$-distortion, asymptotically.

This result is new for the diamond graphs, as the previously best known lower bound for $c_1(\tc(D_n))$ was $\Omega(\sqrt{n})$. Even more, for many $\oslash$ powers $\{G_n\}_{n=1}^\infty$, including the Laakso graphs, it was not even known before this work whether $\sup_n c_1(\tc(G_n)) = \infty$. We provide good estimates for the constants $C$ of Theorem \ref{thm:c1(TC(slash))} in the case of diamond and Laakso graphs in Corollary \ref{C:DiamLaak}.

\subsection{Organization and Description of Proofs}

Our method of proof is the same for each of Theorem~\ref{thm:c1(TC(grids))} and \ref{thm:c1(TC(slash))}: we prove that if some collection of transportation problems $\{\mu_t\}_{t\in \M}\sbs\tc(G_n)$ satisfies a pair of conditions {\bf(C\ref{C1})}-{\bf(C\ref{C2})} (defined in \S\ref{sec:(C1)-(C2)}), then $\tc(G_n)$ must incur a certain amount of distortion when linearly embedded into an $\ell_1^m$ space. This is the content of Theorem~\ref{thm:c1(TC)}, which is proved in \S\ref{ss:(C1)-(C2)}. This general method is new and likely to find future use in estimating the $L_1$-distortion of other transportation cost spaces.

In \S\ref{sec:grids}, we show that the conditions {\bf(C\ref{C1})}-{\bf(C\ref{C2})} are satisfied for the planar grids $\{0,\dots n\}^2$ for a choice of measures $\{\mu_t\}_{t\in \M}$ that have ``cross-shaped" supports (see Figure~\ref{fig:Gr4}). This is the content of Theorems~\ref{thm:(C2)grids} and \ref{thm:(C1)grids}. Theorem~\ref{thm:c1(TC(grids))} follows from these and Theorem~\ref{thm:c1(TC)} (see Theorem~\ref{thm:c1(TC(Gr))2}). In particular, the satisfaction of {\bf(C\ref{C1})}, proved in Theorem~\ref{thm:(C1)grids} in \S\ref{S:C1}, shows that functions $f: \{0,\dots n\}^2 \to \R$ satisfy a newly defined Sobolev-type inequality. This theorem could be of independent interest. The classical endpoint Sobolev inequality states that the $L_2$-variance of $f$ is bounded by the $W^{1,1}$-norm of $f$, and this inequality was an essential ingredient in Naor-Schechtman's \cite{NS07} proof of $c_1(\tc(\{0,\dots n\}^2)) = \Omega(\sqrt{\log n})$ (in fact they adapted to the discrete setting an earlier argument of Kislyakov \cite[Theorem~3]{Kis75} employing the Sobolev inequality). These arguments inspired us to search for a different Sobolev-type inequality that could yield better distortion lower bounds.

In \S\ref{S:EdgeRepl}, we describe the operations of Edge Replacement (Definition \ref{D:EdgeRep}) and Slash Product (Definition \ref{D:SlashPr}). Then, we describe the choice of measures $\{\mu_t\}_{t\in \M}$, and show that they satisfy conditions {\bf(C\ref{C1})}-{\bf(C\ref{C2})} in a rather general situation. Finally, we use the theory developed in \S\ref{sec:(C1)-(C2)} and \S\ref{S:EdgeRepl} to get the $L_1$-distortion estimates for diamond and Laakso graphs (Corollary \ref{C:DiamLaak}) and a proof of Theorem \ref{thm:c1(TC(slash))}, see page \pageref{T:1.2Rep}.

In \S\ref{sec:NSred}, we provide the proof of the reduction to linear maps described in equation \eqref{E:NSred}. 
\medskip

\subsection{The Basic Idea of the Proof}

The techniques used in the proof of the main step, Theorem \ref{thm:c1(TC)}, resemble the techniques used in many lower distortion bounds. We mean the well-known method based on Poincaré inequalities (see \cite[Chapter 4]{Ost13}). Let us remind the basic idea.

Let $(X,d_X)$ and $(Y,d_Y)$ be metric spaces. Suppose that our goal is to estimate the infimal distortion of embeddings of $X$ into $Y$ from below. To achieve this, we find a subset $U$ of $X$, arrays of real numbers $\{a_{uv}\}_{u,v\in U}$ and $\{b_{uv}\}_{u,v\in U}$, and $p\in[1,\infty)$, such that, on one hand,
\begin{equation}\label{E:PoinGenId}
\sum_{u,v\in U} a_{u,v}(d_X(u,v))^p\le\sum_{u,v\in U}b_{u,v}
(d_X(u,v))^p,
\end{equation}
and, on the other hand, for any mapping $f:X\to Y$
\begin{equation}\label{E:PoinGen}
\sum_{u,v\in U} a_{u,v}(d_Y(f(u),f(v)))^p\ge D\sum_{u,v\in
U}b_{u,v}(d_Y(f(u),f(v)))^p.
\end{equation}
The inequalities above are called {\it Poincar\'e inequalities}. It is not difficult to conclude that inequalities \eqref{E:PoinGenId} and \eqref{E:PoinGen} imply that the distortion of any map of $X$ into $Y$ is at least $D^{1/p}$.

We use a similar method. We also observe that because of the linearization equality \eqref{E:NSred}, we may assume that $f$ is linear. Because of this, we may, in addition, assume that $f$ is noncontractive and $\|f\|\le$\,distortion$\,+\ep$.  

To achieve the optimal distortion rate, in our versions of inequalities \eqref{E:PoinGenId} and \eqref{E:PoinGen}, we had to mix exponents $1$ and $2$ on one side. Because of this, we call the obtained inequality `Sobolev inequality' (let us recall that in the classical Poincaré inequality exponents in both sides are the same, but in the classical Sobolev inequality the exponents are mixed).

We work with the case in which the metric space $G$ is a weighted graph, $G=(V,E,w)$, with a positive weight $w$ and the shortest path distance $d_G$ satisfying $d_G(u,v)=w(uv)$ for each edge $uv$. Furthermore, we assume that there is another positive function $th$ on $E$ called {\it thickness}, satisfying 
\[\sum_{uv\in E(G)} th(uv)d_{G_n}(uv) = 1.\]

Our analogues of the Poincar\'e inequalities are the following.  For some graphs $G$ we find index sets $\M=\sqcup_{k=1}^m\M_k$ and systems $\{\mu_t\}_{t\in \M}\subseteq \tc(G)$ of transportation problems, such that for certain matrices $\{\theta_s(t)\}_{s,t\in \M_k}, \theta_s(t)=\pm 1$, and numbers $\alpha$ and $C$ (independent of $m$), we have (here $e^+$ and $e^-$ are endpoint vertices of the oriented edge $e$)
\begin{equation}\label{E:SobId}\sum_{k=0}^m\left(\frac1{|\M_k|}\sum_{s\in \M_k}\left\|\sum_{t\in \M_k}\theta_s(t)\mu_t\right\|^2_{\tc}\right)^{1/2}
\ge m\cdot\alpha\sum_{e\in E}\|\delta_{e^+}-\delta_{e^-}\|_{\tc}th(e),\end{equation}
and on the other hand, for any linear noncontractive map $f:\tc(X)\to \ell_1^d$, we have
\begin{equation}\label{E:SobGen}\sum_{k=0}^m\left(\frac1{|\M_k|}\sum_{s\in \M_k}\left\|f\left(\sum_{t\in \M_k}\theta_s(t)\mu_t\right)\right\|^2_{\ell_1^d}\right)^{1/2}
\le C\sum_{e\in E}\|f(\delta_{e^+}-\delta_{e^-})\|_{\ell_1^d}th(e).\end{equation}
As for the Poincaré inequalities above, it is easy to see that the distortion estimate $c_1(\tc(G))\ge \frac{m\alpha}C$ follows.
Note that in the argument below, it is convenient to use slightly different notation. The inequality \eqref{E:SobId} follows eventually from \eqref{E:TCtheta} and the inequality \eqref{E:SobGen} can be found inside \eqref{E:Show}.

To prove the main results of the paper, it remains to observe that for some families $\{G_n\}_{n=1}^\infty$ of such graphs, the parameter $m=m(n)$ grows as needed, while the quotient $\frac{\alpha(n)}{C(n)}$ stays bounded from below.

\section{Conditions Giving Lower Estimates for the $L_1$-Distortion of Transportation Cost Spaces}\label{sec:(C1)-(C2)}

Our goal in this section is to find conditions which can be used for nontrivial (often tight) estimates from below for the rate of growth of $\{c_1(\TC(G_n))\}_{n=1}^\infty$ for families $\{G_n\}_{n=1}^\infty$ of graphs with weighted geodesic metrics. We write $V(G)$ to denote the vertex set of the graph $G$, $E(G)$ to denote the edge set, $d_G(u,v)$ to denote the distance between $u,v\in V(G)$. We assume that an orientation on $E(G)$ is chosen (sometimes, we shall specify a choice of orientation, if we do not do this, then the choice is irrelevant in the sense that parameters that are essential for our argument do not depend on it), and we write $uv$ or $(u,v)$ to denote the directed edge with initial vertex $u$ and terminal vertex $v$. We write $d_G(uv) = d_G(u,v)$ to denote the length of an edge $uv\in E(G)$. In Section \ref{S:EdgeRepl}, we will need more notation for directed graphs, and will introduce it there.

We start by showing estimates on the rate of growth of  $\{c_1(\TC(G_n))\}_{n=1}^\infty$ for families satisfying the conditions  {\bf(C\ref{C1})}-{\bf(C\ref{C2})} described below with ``good'' parameters. After that, we provide some important examples where the conditions  {\bf(C\ref{C1})}-{\bf(C\ref{C2})} are satisfied with ``good'' parameters. The conditions  {\bf(C\ref{C1})}-{\bf(C\ref{C2})} are on transportation problems $\{\mu_t\}_{t\in \M}\sbs\TC(G_n)$, where $\M$ is some index set with a partition $\M = \sqcup_{k=0}^{n'} \M_k$.

\begin{enumerate}[{\bf(C1)}]
	\item \label{C1} There exists $C<\infty$ such that, for all $n\geq 0$, there exists $th: E(G_n) \to (0,\infty)$, called a {\it thickness function}, such that
	\begin{equation} \label{eq:probabilitymeasure}
		\sum_{uv\in E(G_n)} th(uv)d_{G_n}(uv) = 1 
	\end{equation}
	and all scalar-valued functions $f: V(G_n) \to \R$ satisfy the Sobolev-type inequality
	\begin{equation}\label{E:Sobolev}
		\sum_{k=0}^{n'}\Big(\sum_{t\in \M_k}\left|\int f d\mu_t\right|^2\Big)^{\frac12} \leq C\sum_{uv\in E(G_n)} |f(u)-f(v)|th(uv).
	\end{equation}
	
	

	\item \label{C2} For each $k \in \{0,\dots n'\}$, there exist $\alpha_k \geq 0$ and an orthogonal system of scalar-valued functions $\{\theta_i: \M_k \to \R\}_{i\in \M_k}$ with $|\theta_i(t)|\leq 1$ for every $i,t \in \M_k$ and satisfying the uniform lower bound
	
	\begin{equation}\label{E:TCtheta}
		\min_{i\in\M_k}\Big\|\sum_{t\in \M_k}\theta_i(t)\mu_t\Big\|_{\TC} \geq \alpha_k.
	\end{equation}
	
\end{enumerate}

\subsection{Estimates for $L_1$-Distortion implied by   {\bf(C\ref{C1})}-{\bf(C\ref{C2})}}
\label{ss:(C1)-(C2)}

\begin{theorem} \label{thm:c1(TC)}
	Let $(G_n)_{n=1}^\infty$ be a sequence of graphs satisfying conditions  {\bf(C\ref{C1})}-{\bf(C\ref{C2})}. Then 
	
	\begin{equation*}
		c_1(\TC(G_n)) \geq C^{-1}\sum_{k=0}^{n'}\alpha_k.
	\end{equation*}
	
\end{theorem}

\begin{proof}
	Let $k\in\{0,\dots n'\}$. First, setting $\nu_i = \sum_{t\in \M_k}\frac{\theta_i(t)}{\sqrt{|\M_k|}}\mu_t$ and noting that $\left\{\frac{\theta_i}{\sqrt{|\M_k|}}\right\}_{i\in \M_k}$ is an orthogonal set of functions $\M_k \to \R$ with each function having $\ell_2(\M_k)$-norm at most 1, we have
	\begin{equation}\label{E:Parseval}
		\sum_{i\in \M_k}\Big|\int f d\nu_i\Big|^2 \leq \sum_{t\in \M_k}\Big|\int f d\mu_t\Big|^2,
	\end{equation}
	and \eqref{E:TCtheta} implies the estimate 
	\begin{equation}\label{E:TCnu}
		\min_{i\in\M_k}\|\nu_i\|_{\TC} \geq \alpha_k|\M_k|^{-\frac12}.
	\end{equation}
	We combine \eqref{E:Parseval} and \eqref{E:Sobolev} to obtain
	\begin{equation} \label{E:Sobolev2}
		\sum_{k=0}^{n'}\Big(\sum_{i\in \M_k}\Big|\int f d\nu_i\Big|^2\Big)^{\frac 12} \leq C\sum_{uv\in E(G_n)} |f(u)-f(v)|th(uv).
	\end{equation}
	
	Since the right-hand side of \eqref{E:Sobolev2} is an $L_1$-norm, we can derive a corresponding inequality for linear maps taking values in an $\ell_1$-space simply by summing. Indeed, let $L: \TC(G_n) \to \ell_1^d$ be a linear map. Fix a basepoint $x_0 \in V(G_n)$, and define functions $\{f_j: G_n \to \R\}_{j=1}^d$ by $f_j(x) := L(\delta_x-\delta_{x_0})_j$, so that $L(\nu) = (\int f_j d\nu)_{j=1}^d$ for every $\nu \in \TC(G_n)$ (this identity can be proved using the simple fact that $\{\delta_x-\delta_{x_0}\}_{x\in V(G_n)}$ spans $\TC(G_n)$). Then we have
	\begin{equation}\label{E:Show}\begin{split}
			\sum_{k=0}^{n'}\Big(\frac1{|\M_k|}\sum_{i\in \M_k}\Big\|L\Big(\sum_{t\in\M_k}\theta_i(t)\mu_t\Big)\Big\|_1^2 \Big)^{\frac12}&=\sum_{k=0}^{n'}\Big(\sum_{i\in \M_k}\left\|L(\nu_i)\right\|_1^2 \Big)^{\frac12}\\
            &= \sum_{k=0}^{n'}\Big(\sum_{i\in \M_k}\Big(\sum_{j=1}^d\left|\int f_j d\nu_i\right|\Big)^2\Big)^{\frac12} \\
		&\stackrel{\rm triangle}{\leq} \sum_{j=1}^d\sum_{k=0}^{n'}\Big(\sum_{i\in \M_k}\left|\int f_j d\nu_i\right|^2\Big)^{\frac12} \\
		&\stackrel{\eqref{E:Sobolev2}}{\leq} \sum_{j=1}^d C\sum_{uv\in E(G_n)} |f_j(u)-f_j(v)|th(uv) \\
		&= C\sum_{uv\in E(G_n)} \|L(\delta_u-\delta_v)\|_1 th(uv) \\
		&\leq C\sum_{uv\in E(G_n)} \|L\| d_{G_n}(uv)th(uv) \\
		&\overset{\eqref{eq:probabilitymeasure}}{=} C\|L\|. 
	\end{split}
    \end{equation}
	
	With this $\ell_1$-valued Sobolev inequality in hand, obtaining a lower bound for $c_1(\TC(G_n))$ is immediate with the help of \eqref{E:TCnu}: we let $D$ be any number larger than $\{c_1(\TC(G_n))\}_{n=1}^\infty$, and we suppose that $L: \TC(G_n) \to \ell_1^d$ is a noncontractive linear map with $\|L\| \leq D$. We can find such a map $L$ by the equality \eqref{E:NSred} from \S\ref{sec:intro}.  
	Then, using the above inequality, we get the estimate 
	\begin{align*}
		\sum_{k=0}^{n'}\alpha_k \stackrel{\eqref{E:TCnu}}{\leq} \sum_{k=0}^{n'}\Big(\sum_{i\in \M_k}\left\|\nu_i\right\|^2_{\TC}\Big)^{\frac12} \leq \sum_{k=0}^{n'}\Big(\sum_{i\in \M_k}\left\|L(\nu_i)\right\|^2_1\Big)^{\frac12} \leq C\|L\| \leq C D.
	\end{align*}
	
	Since $D > c_1(\TC(G_n))$ was arbitrary, we get  $c_1(\TC(G_n)) \geq C^{-1}\sum_{k=0}^{n'}\alpha_k$.
\end{proof}

\subsection{Reduction to Simply Connected Sets}
In this subsection, we provide a very useful reduction that allows us to verify Sobolev-type inequalities only for indicator functions of sets that are connected and have a connected complement.

Let $G = (V,E)$ be a connected, finite, directed graph and $A \sbs V$. The {\it edge-boundary} of $A$, denoted $\partial_E A$, is the set of all edges having one endpoint in $A$ and the other in $A^c$. Obviously, $\partial_E A = \partial_E A^c$. Let $\nu$ be a fully supported positive measure on $E$. For $f: V \to \R$, we define the gradient $\nabla f: E \to \R$ by $\nabla(f)(uv) := f(v)-f(u)$. We define the $(1,1)$-Sobolev (semi)norm of $f$ by
$$\|f\|_{W^{1,1}(\nu)} := \|\nabla f\|_{L_1(\nu)} = \sum_{uv\in E} |f(v)-f(u)|\nu(uv).$$
Note that when $f = \one_A$ for some $A \sbs V$, $\|\one_A\|_{W^{1,1}(\nu)} = \nu(\partial_E A)$.

Since $G$ is connected and $\nu$ is fully supported, $\|f\|_{W^{1,1}(\nu)} = 0$ if and only if $f$ is a constant function. Thus, the set of functions $V \to \R$ modulo constant functions becomes a normed space when equipped with the Sobolev norm. We denote this normed space by $W^{1,1}(\nu)$.

To prove the desired reduction, we use a characterization of the extreme points of $B_{W^{1,1}(\nu)}$ found in \cite{Ost05}, where $B_W$ denotes the closed unit ball of a normed space $W$. We refer to a set of vertices $A \sbs V$ as {\it simply connected} if $A$ and $A^c$ are connected. We also use the comment on \cite[bottom of p.~501]{Ost05} (stating that the proof for edge-weighted graphs goes along the same lines).

\begin{theorem}[{\cite[Theorem~7]{Ost05}}] \label{thm:extremepoints}
	Suppose $\nu$ is a positive measure on $E$. Then the equivalence class of a function $f: V \to \R$ is an extreme point of $B_{W^{1,1}(\nu)}$ if and only if $\|f\|_{W^{1,1}(\nu)} = 1$ and there exists a simply connected subset $A \sbs V$ such that the restriction of $f$ to $A$ is a constant function and the restriction of $f$ to $A^c$ is a constant function.
\end{theorem}

\begin{theorem}[Reduction to Simply Connected Sets] \label{thm:iso-to-Sobolev}
	Suppose that $\nu$ is a positive measure on $E$. Let $|||\cdot|||$ be a seminorm on $\R^V$ with $|||\one_V||| = 0$, and let $C \in [0,\infty)$ be a constant. If $|||\one_A||| \leq C \|\one_A\|_{W^{1,1}(\nu)}$ for all simply connected sets $A \sbs V$, then $|||f||| \leq C\|f\|_{W^{1,1}(\nu)}$ for all $f: V \to \R$.
\end{theorem}

\begin{proof} Assume that $|||\one_A||| \leq C \|\one_A\|_{W^{1,1}(\nu)}$ for all simply connected sets $A \sbs V$. Let $f: V \to \R$ be a function. If $f$ is a constant function, then $|||f|||=0$ by assumption, and so the conclusion holds trivially. Assume, then, that $f$ is not constant. By homogeneity of the seminorms on each side of the inequality, we may then assume that $\|f\|_{W^{1,1}(\nu)} = 1$. By Carathéodory's convex hull theorem, there exist finitely many extreme points $\{g_i\}_{i=1}^m \sbs B_{W^{1,1}(\nu)}$ and positive scalars $\{c_i\}_{i=1}^m \sbs[0,1]$ such that $\sum_{i=1}^m c_i = 1$ and $f \equiv \sum_{i=1}^m c_ig_i$ (modulo constant functions). By Theorem~\ref{thm:extremepoints}, there exist simply connected subsets $\{A_i\}_{i=1}^m$ of $V$ and scalars $\{a_i\}_{i=1}^m \sbs \R$ such that $g_i \equiv a_i\one_{A_i}$ (modulo constant functions) for each $i\in\{1,\dots m\}$. Then
	$f \equiv \sum_{i=1}^m c_ia_i\one_{A_i}$ (modulo constant functions). Then, since both the seminorm $|||\cdot|||$ and the Sobolev norm vanish on constant functions, we have
	\begin{align*}
		|||f||| &= |||\sum_{i=1}^m c_ia_i\one_{A_i}||| \leq \sum_{i=1}^m c_i|a_i|\,|||\one_{A_i}||| \\
		&\leq C\sum_{i=1}^m c_i|a_i|\,\|\one_{A_i}\|_{W^{1,1}(\nu)} = C\sum_{i=1}^m c_i\|g_i\|_{W^{1,1}(\nu)} = C = C\|f\|_{W^{1,1}(\nu)}.\qedhere
	\end{align*}
\end{proof}

\section{Planar Grids} \label{sec:grids}
Let $n\in\N$ with $n \geq 2$. The {\it dyadic planar grid} is the graph $Gr_n$ with vertex set $V(Gr_n) := \{0,\dots 2^n\}^2$ and directed edge set $E(Gr_n) := \{uv \in V(Gr_n)^2: v \in \{u+(1,0),u+(0,1)\}\}$. Note that $|E(Gr_n)| = 2^{n+1}(2^n+1) \leq \frac{5}{2}4^{n}$. We equip $V(Gr_n)$ with the unit edge length shortest path metric, which coincides with the $\ell_1$-metric on $\R^2$. The diameter function $\diam$ defined on nonempty subsets of a metric space and the distance function $\dist$ defined on pairs of nonempty subsets of a metric space are defined in the usual way.

\subsection{Constructing the Measures}
For the grids, we take $n' = n-2$ and we take the index set $\M$ to be $T = \cup_{k=0}^{n-2} T_k$, where $T_k$ is the Cartesian power $\{1,2,3,4\}^k$, with the convention that $\{1,2,3,4\}^0 = \{\emptyset\}$. We think of $T$ as the regular rooted tree of depth $n' = n-2$ and branching number $4$. When $t \in T$, we define the {\it generation} of $t$ to be the unique integer $|t|$ for which $t \in T_{|t|}$. We let $\{\mu_t\}_{t\in T}$ be any collection of signed measures with 0 total mass in $\TC(Gr_n)$ satisfying the following five properties {\bf (P\ref{P1})}-{\bf (P\ref{P5})} for some $\frac{1}{2} \leq C_1,C_2,C_3,C_4 < \infty$ and every $t \in T$. Recall that the support of a function $f: V \to \R$ on a finite set $V$ is the subset $\supp(f) := \{v \in V: f(v) \neq 0\}$, the support of a signed measure $\mu$ on $V$ is the support of the function $v \mapsto \mu(\{v\})$, and $a \vee b$ denotes the maximum of $a,b\in\R$.

\begin{enumerate}[{\bf(P1)}]
	\item \label{P1} $\diam(\supp(\mu_t)) \leq C_1 2^{n-|t|}$.
	\item \label{P2} If $t' \in T \setminus \{t\}$, then $\dist(\supp(\mu_t),\supp(\mu_{t'})) \geq C_2^{-1}2^{n - |t|\vee|t'|}$.
	\item \label{P3} $\|\mu_t\|_{\TC} \geq C_3^{-1}4^{-|t|}$.
	\item \label{P4} For any $A \sbs V(Gr_n)$, it holds that $|\mu_t(A)| \leq C_4\min\{2^{-n-1-|t|}, 4^{-n}(\diam(A)+1)\}$.
	\item \label{P5} There exists $u \in V(Gr_n)$ such that $\supp(\mu_t) \cup \{u\}$ is connected.
\end{enumerate}

There are many ways to construct measures in $\TC(Gr_n)$ satisfying these five properties. We sketch one particular way in the next theorem, where the supports of the measures are ``cross-shaped".

\begin{theorem} \label{thm:Cvalues}
	There exist measures $\{\mu_t\}_{t\in T} \sbs \TC(Gr_n)$ satisfying {\bf (P\ref{P1})}-{\bf (P\ref{P5})} with $C_1 = \frac{1}{2}$, $C_2 = 4$, $C_3 = 16$, and $C_4 = 1$.
\end{theorem}

\begin{proof} Let $k \in \{0,\dots n-2\}$. We can decompose $Gr_n$ into $4^{k}$ dyadic subgrids, with the collection of subgrids indexed by $T_k$. For $n=4$ and $k=1,2$, this is shown in Figure~\ref{fig:Gr4}. Let $t \in T_k$. Consider the dyadic subgrid $Q_t$ of size $2^{n-k} \times 2^{n-k}$ in $Gr_n$ corresponding to $t$. We draw a vertical black line, with length half that of the side length of $Q_t$, through the center of $Q_t$, so that it is a path subgraph with $2^{n-k-1}+1$ vertices and $2^{n-k-1}$ edges. Then we draw a horizontal red line, with the same length as the black vertical line, through the center of $Q_t$. These two lines intersect exactly at the center of $Q_t$. We define positive measures $\mu_t^+,\mu_t^-$ on $V(Gr_n)$ by $\mu_t^+(\{u\}) = 4^{-n}$ if $u$ belongs to the vertical black line and $\mu_t^+(\{u\}) = 0$ otherwise, and $\mu_t^-(\{u\}) = 4^{-n}$ if $u$ belongs to the horizontal red line and $\mu_t^-(\{u\}) = 0$ otherwise. We define $\mu_t := \mu_t^+ - \mu_t^-$. Thus, $\mu_t$ is a signed measure with 0 total mass and a cross-shaped support. This support is almost the union of the black and red lines, except that the center point where they intersect fails to be in the support. Therefore, the measure $\mu_t$ satisfies {\bf (P\ref{P5})}. See Figure~\ref{fig:Gr4} for a picture of the supports of $\mu_t$ for $|t|\leq 2$ when $n=4$.
	
	We now consider the other properties {\bf (P\ref{P1})}-{\bf (P\ref{P4})}. Let $t \in T$ and set $k := |t|$. It is immediate that {\bf (P\ref{P1})} is satisfied with $C_1 =\frac{1}{2}$.
    
	To see that {\bf (P\ref{P2})} holds with $C_2 = 4$, let $t' \in T \setminus\{t\}$, and assume without loss of generality that $|t'| \leq |t|$. Then $\mu_t$ is supported in the center of a dyadic subgrid $Q_t$ of size $2^{n-k} \times 2^{n-k}$, and the distance from $\supp(\mu_t)$ to the boundary of $Q_t$ is $2^{n-k-2}$. The support of $\mu_{t'}$ is disjoint from the interior of $Q_t$. Hence, the distance between $\supp(\mu_t)$ and $\supp(\mu_{t'})$ is at least $2^{n-k-2}$.
	
	For {\bf (P\ref{P3})}, to obtain a lower bound on $\|\mu_t\|_{\TC}$, we observe that, in any transportation plan from $\mu_t^+$ between $\mu_t^-$, each point $u$ in the support of $\mu_t^+$ must be transported a distance at least the distance from $u$ to the horizontal red line supporting $\mu_t^-$. This leads to a lower bound for $\|\mu_t\|_{\TC}$ of $(4^{-n})2\sum_{i=1}^{2^{n-k-2}} i = 4^{-n}(4^{n-k-2}+2^{n-k-2}) > 4^{-k-2}$. This shows that $C_3 \leq 16$.
	
	Finally, for {\bf (P\ref{P4})}, let $A \sbs V(Gr_n)$. Without loss of generality, assume that $\mu_t^+(A) \geq \mu_t^-(A)$, so that we have $|\mu_t(A)| \leq \mu_t^+(A)$. We have the trivial bound $\mu_t^+(A) \leq \mu_t^+(V(Gr_n)) = 4^{-n}2^{n-k-1} = 2^{-n-k-1}$. Let $y^-$ denote the minimum of $y$-coordinates among all points in $\supp(\mu_t^+) \cap A$ and $y^+$ the maximum $y$-coordinate. Then it is easy to see that $\mu_t^+(A) \leq 4^{-n}(y^+-y^-+1) \leq 4^{-n}(\diam(A)+1)$. Hence $C_4 \leq 1$. 
\end{proof}

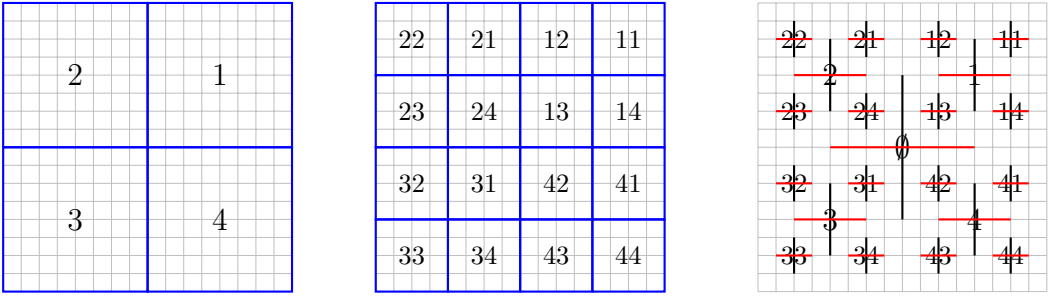
\begin{figure}
	\centering
	\begin{minipage}{0.3\textwidth}
		\centering
		\begin{tikzpicture}[scale=0.24]
			\draw[step=1, gray!50, very thin] (0,0) grid (16,16);
			\foreach \x/\y/\label in {8/8/1, 0/8/2, 0/0/3, 8/0/4} {
				\draw[blue, thick] (\x,\y) rectangle ({\x+8},{\y+8});
				\node[font=\large] at ({\x+4},{\y+4}) {\label};
			}
		\end{tikzpicture}
	\end{minipage}%
	\begin{minipage}{0.3\textwidth}
		\centering
		\begin{tikzpicture}[scale=0.24]
			\draw[step=1, gray!50, very thin] (0,0) grid (16,16);
			\foreach \qx/\qy/\q in {8/8/1, 0/8/2, 0/0/3, 8/0/4} {
				\foreach \sx/\sy/\s in {4/4/1, 0/4/2, 0/0/3, 4/0/4} {
					\pgfmathsetmacro{\x}{\qx+\sx}
					\pgfmathsetmacro{\y}{\qy+\sy}
					\pgfmathsetmacro{\blocklabel}{int(\q*10 + \s)}
					\draw[blue, thick] (\x,\y) rectangle ({\x+4},{\y+4});
					\node[font=\small] at ({\x+2},{\y+2}) {\blocklabel};
				}
			}
		\end{tikzpicture}
	\end{minipage}
	\begin{minipage}{0.3\textwidth}
		\centering
		\begin{tikzpicture}[scale=0.24]
			\draw[step=1, gray!50, very thin] (0,0) grid (16,16);
			
			\node[font=\large] at (8,8) {$\emptyset$};
			\draw[black, thick] (8,4) -- (8,12); 
			\draw[red, thick] (4,8) -- (12,8);   
			
			
			\node[font=\large] at (12,12) {1};
			\draw[black, thick] (12,10) -- (12,14); 
			\draw[red, thick] (10,12) -- (14,12);   
			
			\node[font=\large] at (4,12) {2};
			\draw[black, thick] (4,10) -- (4,14); 
			\draw[red, thick] (2,12) -- (6,12);   
			
			\node[font=\large] at (4,4) {3};
			\draw[black, thick] (4,2) -- (4,6); 
			\draw[red, thick] (2,4) -- (6,4);   
			
			\node[font=\large] at (12,4) {4};
			\draw[black, thick] (12,2) -- (12,6); 
			\draw[red, thick] (10,4) -- (14,4);    

			
			
			\node[font=\small] at (14,14) {11};
			\draw[black, thick] (14,13) -- (14,15); 
			\draw[red, thick] (13,14) -- (15,14);   
			
			\node[font=\small] at (10,14) {12};
			\draw[black, thick] (10,13) -- (10,15); 
			\draw[red, thick] (9,14) -- (11,14);   
			
			\node[font=\small] at (10,10) {13};
			\draw[black, thick] (10,9) -- (10,11); 
			\draw[red, thick] (9,10) -- (11,10);   
			
			\node[font=\small] at (14,10) {14};
			\draw[black, thick] (14,9) -- (14,11); 
			\draw[red, thick] (13,10) -- (15,10);   

			
			\node[font=\small] at (6,14) {21};
			\draw[black, thick] (6,13) -- (6,15); 
			\draw[red, thick] (5,14) -- (7,14);   
			
			\node[font=\small] at (2,14) {22};
			\draw[black, thick] (2,13) -- (2,15); 
			\draw[red, thick] (1,14) -- (3,14);   
			
			\node[font=\small] at (2,10) {23};
			\draw[black, thick] (2,9) -- (2,11); 
			\draw[red, thick] (1,10) -- (3,10);   
			
			\node[font=\small] at (6,10) {24};
			\draw[black, thick] (6,9) -- (6,11); 
			\draw[red, thick] (5,10) -- (7,10);   

			
			\node[font=\small] at (14,6) {41};
			\draw[black, thick] (14,5) -- (14,7); 
			\draw[red, thick] (13,6) -- (15,6);   
			
			\node[font=\small] at (10,6) {42};
			\draw[black, thick] (10,5) -- (10,7); 
			\draw[red, thick] (9,6) -- (11,6);   
			
			\node[font=\small] at (10,2) {43};
			\draw[black, thick] (10,1) -- (10,3); 
			\draw[red, thick] (9,2) -- (11,2);   
			
			\node[font=\small] at (14,2) {44};
			\draw[black, thick] (14,1) -- (14,3); 
			\draw[red, thick] (13,2) -- (15,2);   

			
			\node[font=\small] at (6,6) {31};
			\draw[black, thick] (6,5) -- (6,7); 
			\draw[red, thick] (5,6) -- (7,6);   
			
			\node[font=\small] at (2,6) {32};
			\draw[black, thick] (2,5) -- (2,7); 
			\draw[red, thick] (1,6) -- (3,6);   
			
			\node[font=\small] at (2,2) {33};
			\draw[black, thick] (2,1) -- (2,3); 
			\draw[red, thick] (1,2) -- (3,2);   
			
			\node[font=\small] at (6,2) {34};
			\draw[black, thick] (6,1) -- (6,3); 
			\draw[red, thick] (5,2) -- (7,2);   
		\end{tikzpicture}
	\end{minipage}%
	
	\caption{Left: The decomposition of $Gr_4$ into dyadic subgrids $\{Q_t\}_{t\in T_1}$ indexed by $T_1 = \{1,2,3,4\}$. Center: The decomposition of $Gr_4$ into dyadic subgrids $\{Q_t\}_{t\in T_2}$ indexed by $T_2 = \{1,2,3,4\}^2$. Right: The positive (black) and negative (red) supports of the measures $\mu_t$ for $|t| \leq 2$.}
	\label{fig:Gr4}
\end{figure}

\subsection{Proof of Condition {\bf (C\ref{C2})}}
We begin with a simple lemma stating that the transportation cost norm is approximately additive when the supports of the measures involved are sufficiently separated. This is intuitively clear, and one will frequently encounter closely related estimates of this type in the transportation cost literature. Once we've established the lemma, condition {\bf (C\ref{C2})} will quickly follow.

\begin{lemma} \label{lem:TCadditive}
	Let $(X,d)$ be a finite metric space and $M$ a finite collection of nonzero signed measures on $X$, each having 0 total mass, such that
	\begin{itemize}
		\item $\sup_{\mu\in M}\diam(\supp(\mu)) \leq A$ and
		\item $\inf_{\mu',\mu\in M,\mu'\neq\mu} \dist(\supp(\mu),\supp(\mu')) \geq B^{-1}$
	\end{itemize}
	for some $A,B<\infty$. Then for any function $g: M \to \R$,
	\begin{equation*}
		\Big\|\sum_{\mu\in M}g(\mu)\mu\Big\|_{\TC} \geq (AB)^{-1}\sum_{\mu\in M}|g(\mu)|\|\mu\|_{\TC}.
	\end{equation*}
\end{lemma}

\begin{proof} Denoting $g(\mu)\mu$ by $\tau$, we can write the desired inequality as 
	\begin{equation}\label{E:TCcomb}
		(AB)\left\|\sum_{\tau\in M'}\tau\right\|_{\TC}\ge \sum_{\tau\in M'}\|\tau\|_{\TC}   
	\end{equation}
	Let
	\[\sum_{\tau\in M'}\tau =a_1(\delta_{x_1}-\delta_{y_1})+a_2(\delta_{x_2}-\delta_{y_2})+\dots+ a_n(\delta_{x_n}-\delta_{y_n})\]
	be an optimal transportation plan (recall that our space is finite). It is easy to see that we may assume that $\{x_i,y_i\}_{i=1}^n$ are contained in $\cup_{\tau\in M'}\supp(\tau)$. If each pair $(x_i,y_i)$ is inside the support of some $\tau\in M'$, the inequality \eqref{E:TCcomb} is immediate with $AB$ replaced by $1$.
	
	If some pairs $(x_i,y_i)$ are split and their elements belong to the supports of different $\tau$'s, we change the transportation plan in the following way. For each $\tau$, let $P_\tau=\{i:~x_i\in\supp(\tau), y_i\notin\supp(\tau)\}$ and $N_\tau=\{i:~y_i\in\supp(\tau), x_i\notin\supp(\tau)\}$. Since $\tau(X)=0$, we get that $\sum_{i\in P_\tau}a_i=\sum_{i\in N_\tau}a_i$. Thus, we can match the corresponding availabilities and needs, and obtain a modified plan that, for each $\tau$, matches its availabilities with its needs. The cost of the resulting plan is at most a factor of $AB$ times the cost of the original plan, because each distance $d(x_i,y_i)$ from the supports of different $\tau$ is at least $B^{-1}$, and from the support of the same $\tau$ is at most $A$. On the other hand, the cost of the resulting plan cannot be less than the sum of optimal plans for each of $\tau$. The inequality \eqref{E:TCcomb} follows.
\end{proof}



\begin{theorem} \label{thm:(C2)grids}
	For any $k \in \{0,\dots n-2\}$ and measures $\{\mu_t\}_{t\in T_k} \sbs \TC(Gr_n)$ satisfying {\bf (P\ref{P1})}-{\bf (P\ref{P3})}, there exists an orthogonal system of functions $\{\theta_i: T_k \to \{-1,1\}\}_{i\in T_k}$ such that inequality \eqref{E:TCtheta} holds with $\alpha_k = (C_1C_2C_3)^{-1}$.
\end{theorem}

\begin{proof}
	Since $|T_k| = 4^k$ is a power of 2, we can find an orthogonal system of functions $\{\theta_i: T_k \to \{-1,1\}\}_{i\in T_k}$ by Sylvester's construction of Hadamard matrices \cite[$\S$~2.1.1]{Hor07} (also known as Walsh matrices). The inequality \eqref{E:TCtheta} with $\alpha_k = (C_1C_2C_3)^{-1}$ immediately follows for this system by properties {\bf (P\ref{P1})}-{\bf (P\ref{P3})}, Lemma~\ref{lem:TCadditive}, and the fact that $|T_k| = 4^{k}$.
\end{proof}

\subsection{Proof of Condition {\bf (C\ref{C1})}}\label{S:C1}
The thickness function we take on $E(Gr_n)$ is the uniform probability measure; $th(uv) = \frac{1}{|E(Gr_n)|}$ for all $uv\in E(Gr_n)$. Theorem \ref{thm:iso-to-Sobolev} implies that in order to prove the Sobolev inequality \eqref{E:Sobolev} for a general function $f: V(Gr_n) \to \R$, it suffices to prove it in 
the case $f = \one_A$ for some connected $A \sbs V(Gr_n)$ with connected complement, also known as a {\it simply connected} set.
To prove the Sobolev inequality for such indicator functions, we need to make some observations on graphs and planar grids. 

\subsubsection{Graph and Planar Topology}
We now collect some basic facts about general graphs, as well as facts specific to the topology of planar grids. Since these latter facts can fail quite badly in graphs such as trees and cycles (whose transportation cost spaces do embed into $L_1$ with uniformly bounded distortion), they can be seen as important metric-topological reasons for the distortion $c_1(\TC(Gr_n))$ being large.
We begin with the elementary facts holding in any graph. 

\begin{proposition} \label{prop:separated}
	Let $G = (V,E)$ be any finite unweighted graph, $d$ the associated shortest path metric in $G$, and $\diam$ the associated diameter, allowing the value $\infty$ in the event that $G$ is not connected. Let $A \sbs V$ be connected. Then the following are true.
	\begin{enumerate}[{\rm (1)}]
		\item $\diam(A)+1 \leq |A|$.
		\item If there exist $r > 0$ and an $r$-separated subset $S \sbs A$ (meaning $d(u,v) \geq r$ for all $u,v\in S, u\neq v$) with $|S| \geq 2$, then $r|S| \leq 3|A|$.
	\end{enumerate}
\end{proposition}

We do not claim that the constant 3 in the second item above is optimal. This proposition plays a crucial role in the proof of Theorem~\ref{thm:(C1)grids}, by way of \eqref{E:isoperimetric}.

\begin{proof}[Proof of Proposition~\ref{prop:separated}]
	The first item is even more obvious than the second: the cardinality of $A$ bounds the cardinality of any chosen path in $A$ (which exists since $A$ is connected), and a path can be chosen such that its cardinality bounds $\diam(A)+1$.
	
	We now prove the second item. For each $s \in S$, let $B_{r/3}(s)$ denote the subset of $A$ consisting of all those $a \in A$ that can be joined to $s$ by a path that is completely contained in $A$ and that has length at most $r/3$. Since $S$ is $r$-separated, the triangle inequality implies that the sets $\{B_{r/3}(s)\}_{s\in S}$ are pairwise disjoint. Moreover, we have the cardinality bound $|B_{r/3}(s)| \geq r/3$. This is due to the following fact: since $|S| \geq 2$ and $A$ is connected, for any $s \in S$, we can choose $s_0 \in S \setminus \{s\}$ and a path $P \sbs A$ joining $s$ to $s_0$. We must have that $|P| \geq r$, and thus $|P \cap B_{r/3}(s)| \geq r/3$. This cardinality bound together with the pairwise disjointness implies $r/3 \cdot |S| \leq |A|$.
\end{proof}

For a graph $G = (V,E)$, we define the {\it symmetric vertex-boundary} of $A$ as $\partial_V A := \bigcup_{uv\in\partial_E A} \{u,v\}$.
Obviously, $\partial_V A = \partial_V A^c$ and
\begin{equation} \label{eq:EbndryVbndry}
	|\partial_V A| \leq 2|\partial_E A|.
\end{equation}

We now give the two essential lemmas concerning the topology of planar grids and sketch their proofs. The first is the graph-theoretical version of the classical result from planar topology that simply connected subsets of the 2-dimensional sphere have connected boundary. Recall that when $A$ is a subset of the vertex set of a graph, we have defined $A$ to be simply connected if $A$ and $A^c$ are both connected.

\begin{lemma} \label{lem:connectedboundary}
	If $A \sbs V(Gr_n)$ is simply connected, then $\partial_V A$ is connected.
\end{lemma}

\begin{proof}
For each $x = (x_1,x_2)\in V(Gr_n)$, let $\bar{x} := [x_1-\tfrac{1}{2},x_1+\tfrac{1}{2}]\times [x_2-\tfrac{1}{2},x_2+\tfrac{1}{2}] \sbs \R^2$, i.e., $\bar{x}$ is the closed $1\times 1$ square centered at $x$. For $A\sbs V(Gr_n)$, let $\tilde{A} := \intr\left(\bigcup_{x\in A}\bar{x}\right) \sbs \R^2$, where $\intr(U)$ denotes the interior of $U$ as a subset of the topological space $\R^2$.
We also introduce $\tilde{A^c} := \intr\left(\bigcup_{x\in A^c}\bar{x}\right) \sbs \R^2$.

We get two disjoint subsets of the square $S:=[-\tfrac{1}{2},2^n+\tfrac{1}{2}]\times [-\tfrac{1}{2},2^n+\tfrac{1}{2}]$. It is clear that the only parts of the square $S$ which are not in $\tilde{A^c}\cup \tilde A$ are (1) The boundary of $S$, (2) The closed line segments between the squares belonging to different regions, $\tilde{A^c}$ and $\tilde{A}$.

We form a graph $B$ in which the closed line segments between the squares belonging to different regions, $\tilde{A^c}$ and $\tilde{A}$, are edges, and their endpoints are vertices. It is important to observe that the degrees of the vertices in this graph can be either $1$, $2$, or $4$, and the degree $1$ occurs if and only if the vertex is on the boundary of $S$.
The statement about the vertices in the boundary of $S$ is clear. Let us consider a vertex that is not in the boundary of $S$. Edges from it are perpendicular to the line segments joining the vertices of the grid belonging to different sets of $\tilde{A^c}$ and $\tilde{A}$. This is why the degree of the vertices of $B$, which are not in the boundary of $S$, can be either $2$ or $4$.

By a version of the Euler theorem (on Euler tours and trails, see \cite[Section 3.3 and Exercise 3.3.5]{BM08}), the graph $B$ is an edge-disjoint union of paths and cycles, and paths can only join different vertices of the boundary. 

Now we consider the graph $\widehat{B}$, in which, in addition to $B$, we include all line segments of the boundary of $S$. We get a planar graph. By the Euler formula, $n-m+\ell=c+1$, where $n$ is the number of vertices, $m$ is the number of edges, $\ell$ is the number of faces, and $c$ is the number of connected components.  Since the graph has three faces, $\tilde{A^c}$, $\tilde{A}$, and the outer face, we get $n-m=c-2$. 

Consider two cases: (a) There is a path among the paths and cycles, (b) There are only cycles.

In both cases, the planar graph obtained by picking one path in case (a), one cycle in case (b), and omitting further paths and cycles already has three faces. Therefore, in order to show that there are no further paths and cycles, it suffices to show that adding any of them increases the number of faces.

Case (a). Adding a path of length $k$ adds at most $k-1$ vertices, $k$ edges, and does not change the number of components of $\widehat{B}$; therefore, it increases the number of faces. Adding a cycle. There are two subcases, in one of them the cycle is a new component in $\widehat{B}$. In this case, we add $k$ vertices, $k$ edges, and $1$ component, so the number of faces has to increase by $1$. In the second subcase, we add more edges than vertices without changing the number of components, so the number of faces should increase.

Case (b). We need consider adding a cycle only, but this can be done exactly as in the previous paragraph.

The last step: let $e_1,\dots, e_n$ be edges of the only path/cycle in the partition of $B$ described above, in their natural order. Then edges of $Gr_n$ crossing $e_k$ and $e_{k+1}$ either have a common vertex (if $e_k$ and $e_{k+1}$ are perpendicular) or their vertices are vertices of a grid square (if the edges are on a straight line). All of these vertices are in $\partial_V A$. In either case, we get connectivity.
\end{proof}    

\begin{lemma} \label{lem:diambound}
	For any $A \sbs V(Gr_n)$ with $A,A^c \neq \emptyset$, we have $\min\{\diam(A),\diam(A^c)\} \leq 2\diam(\partial_V A)$.
\end{lemma}

\begin{proof} Color vertices of $A$ blue and of $A^c$ yellow. The result is immediate if there are two opposite multicolored sides of the grid. In fact, in this case, $\diam(\partial_V A)=2^n$. On the other hand, it is clear that $\max\{\diam(A),\diam(A^c)\} \leq 2^{n+1}$.

So, assume that each pair of opposite sides has one unicolored side, say, $S_1$ for one pair and $S_2$ for the other pair. Since $S_1$ and $S_2$ have a common vertex, the color of $S_1$ and $S_2$ should be the same. Assume without loss of generality that $S_1$ and $S_2$ are both blue. 

Now, consider the smallest rectangle containing all yellow vertices. Let $m$ be its longest side. Consider the two opposite shortest sides of this rectangle. They contain yellow vertices, and their straight-line extensions necessarily contain blue vertices. Therefore $\diam(\partial_V A)\ge m$. On the other hand, since $A^c$ is contained in this rectangle, we have $\diam (A^c)\le 2m$.
\end{proof}

\subsubsection{Proof of the Sobolev Inequality for Simply Connected Sets} We now utilize Lemmas~\ref{lem:connectedboundary} and \ref{lem:diambound} to prove the Sobolev inequality
\begin{equation} \label{E:isoperimetric}
	|||\one_A||| \leq C \|\one_A\|_{W^{1,1}(th)}
\end{equation}
for all simply connected sets $A \sbs V(Gr_n)$, where $th$ is the thickness, which in the case of $Gr_n$ is the uniform probability measure on $E(Gr_n)$, the constant $C$ is bounded by $C_4(20+40C_2)$ (where $C_2,C_4$ are the data in {\bf (P\ref{P2})}, {\bf (P\ref{P4})}, and the seminorm is given by
$$|||f||| := \sum_{k=0}^{n-2}\|f\|_k$$
with
$$\|f\|_k := \Big(\sum_{t\in T_k}\left|\int f d\mu_t\right|^2\Big)^{\frac12},$$
where $\{\mu_t\}_{t\in T}$ are signed measures with 0 total mass satisfying properties {\bf (P\ref{P2})}, {\bf (P\ref{P4})}, and {\bf (P\ref{P5})}. Once \eqref{E:isoperimetric} has been proved, we apply Theorem~\ref{thm:iso-to-Sobolev}, and immediately get the following theorem, establishing condition {\bf (C\ref{C1})}.

\begin{theorem} \label{thm:(C1)grids}
	For all $\{\mu_t\}_{t\in T} \sbs \TC(Gr_n)$ satisfying {\bf (P\ref{P2})}, {\bf (P\ref{P4})}, and {\bf (P\ref{P5})},  and for all functions $f: V(Gr_n) \to \R$,
	\begin{equation*}
		\sum_{k=0}^{n-2}\Big(\sum_{t\in T_k}\left|\int f d\mu_t\right|^2\Big)^{\frac12} \leq C\sum_{uv\in E(G_n)} |f(u)-f(v)|th(uv),
	\end{equation*}
	where $C \leq C_4(20+40C_2)$.
\end{theorem}

First we need a short lemma that is an important consequence of property {\bf (P\ref{P5})}.

\begin{lemma} \label{lem:intersect}
	For every $A \sbs V(Gr_n)$ and $t \in T$, if $\mu_t(A) \neq 0$, then $\partial_V A \cap \supp(\mu_t) \neq \emptyset$.
\end{lemma}

\begin{proof}
	Let $A \sbs V(Gr_n)$ and $t \in T$. Assume that $\mu_t(A) \neq 0$. Since $\mu_t$ has 0 total mass, it cannot happen that $\supp(\mu_t) \sbs A$ or $\supp(\mu_t) \sbs A^c$. Thus, $\supp(\mu_t) \cap A \neq \emptyset$ and $\supp(\mu_t) \cap A^c \neq \emptyset$. Since $\supp(\mu_t) \cup \{u\}$ is connected for some $u \in V(Gr_n)$ by {\bf (P\ref{P5})}, it is easy to see that the previous sentence implies that there exist $x,y \in \supp(\mu_t) \cup \{u\}$ such that $(x,y) \in \partial_E A$. By definition, this implies $x,y \in \partial_V A$. Since it cannot happen that both $x$ and $y$ equal $u$, at least one of $x,y$ belongs to $\supp(\mu_t)$, showing $\partial_V A \cap \supp(\mu_t) \neq \emptyset$.
\end{proof}

We now prove the Sobolev inequality for simply connected sets.

\begin{proof}[Proof of \eqref{E:isoperimetric}]
	Let $A \sbs V(Gr_n)$ be simply connected. Notice that the terms in the sum on the left-hand side of \eqref{E:isoperimetric} are
	$$\|\one_A\|_k = \Big(\sum_{t\in T_k}|\mu_t(A)|^2\Big)^{\frac12},$$
	and that the norm on the right-hand side is
	\begin{equation} \label{eq:rhsGr}
		\|\one_A\|_{W^{1,1}(th)} = \frac{|\partial_E A|}{|E(Gr_n)|} \geq \frac{2}{5}4^{-n}|\partial_E A|.
	\end{equation}
	It is clear that each side of \eqref{E:isoperimetric} is unchanged if we replace $A$ with $A^c$, so it suffices to assume that $A,A^c \neq \emptyset$ and that $\diam(A) \leq \diam(A^c)$.
	
	Choose $m\in\N$ such that
	\begin{equation} \label{eq:|Vboundary|}
		2^{m-1} \leq |\partial_V A| < 2^{m},
	\end{equation}
	which implies by \eqref{eq:EbndryVbndry} that
	\begin{equation} \label{eq:|Eboundary|}
		2^{m-2} \leq |\partial_E A|.
	\end{equation}
	and hence by \eqref{eq:rhsGr} and \eqref{eq:|Eboundary|},
	\begin{equation} \label{eq:4^(-n)2^m}
		4^{-n}2^m \leq 10\|\one_A\|_{W^{1,1}(th)}.
	\end{equation}
	By Lemma~\ref{lem:diambound}, Lemma~\ref{lem:connectedboundary}, and Proposition~\ref{prop:separated}(1), we have that
	\begin{equation} \label{eq:diam(A)}
		\diam(A)+1 \leq 2\diam(\partial_V A)+1 < 2|\partial_V A| \overset{\eqref{eq:|Vboundary|}}{\leq} 2^{m+1} \overset{\eqref{eq:|Eboundary|}}{\leq} 8|\partial_E A|.
	\end{equation}
	Of course, in order to prove the desired inequality, we only need to consider those $t\in T$ with $\mu_t(A) \neq 0$. We proceed to examine those elements $t \in T$ for which this can happen.
	
	Consider the constant $C_2$ from property {\bf (P\ref{P2})}. Let $m_2 \in \Z$ such that
	\begin{equation} \label{eq:m2def}
		2^{m_2-1} \leq C_2 < 2^{m_2}.
	\end{equation}
	We claim that there is at most one element $t \in \cup_{k=0}^{n-m-m_2}T_k$ such that $\mu_t(A) \neq 0$. Indeed, suppose towards a contradiction that we have $t,t' \in T$ with $|t|,|t'| \leq n-m-m_2$, $t\neq t'$, and $\mu_t(A),\mu_{t'}(A) \neq 0$. By Lemma~\ref{lem:intersect}, the intersections $\partial_V A \cap \supp(\mu_t)$ and $\partial_V A \cap \supp(\mu_{t'})$ are both nonempty. Then by \eqref{eq:diam(A)}, {\bf(P\ref{P2})}, and \eqref{eq:m2def}, we have
	\begin{equation*}
		2^{m} \geq \diam(\partial_V A) \geq C_2^{-1}2^{n-(n-m-m_2)} > 2^{-m_2}2^{m+m_2},
	\end{equation*}
	a contradiction. This proves the claim.
	
	The claim together with {\bf (P\ref{P4})} obviously imply that
	\begin{equation*}
		\sum_{k=0}^{n-m-m_2} \|\one_A\|_k \leq C_4 4^{-n}(\diam(A)+1),
	\end{equation*}
	and then by \eqref{eq:diam(A)} and \eqref{eq:rhsGr}, we get that
	\begin{equation} \label{eq:smallkestimate}
		\sum_{k=0}^{n-m-m_2} \|\one_A\|_k \leq C_4 4^{-n}8|\partial_E A| \leq 20C_4\|\one_A\|_{W^{1,1}(th)}.
	\end{equation}
	
	Now we want an upper bound for the term $\|\one_{A}\|_k$ that will be sufficient for larger values of $k$. Fix $k \in \{0,\dots n-2\}$. We first bound the cardinality of the set $\{t \in T_k: \mu_t(A) \neq 0\}$. By Lemma ~\ref{lem:intersect}, this is a subset of $\{t\in T_k: \partial_V A \cap \supp(\mu_t) \neq \emptyset\}$. There are two cases to consider: (i) the cardinality of this set is at most 1, and (ii) the cardinality of this set is at least 2. In case (i), we have the cardinality bound that we want by assumption. In case (ii), the hypothesis of Proposition~\ref{prop:separated}(2) is met, and thus we have by {\bf (P\ref{P2})}, Lemma~\ref{lem:connectedboundary}, Proposition~\ref{prop:separated}(2), and \eqref{eq:|Vboundary|} that
	\begin{equation*}
		|\{t\in T_k: \partial_V A \cap \supp(\mu_t) \neq \emptyset\}| \leq 3C_22^{k-n}|\partial_V A| \leq 3C_22^{m+k-n}.
	\end{equation*}
	Therefore, in both cases, we have  that
	\begin{equation*}
		|\{t \in T_k: \mu_t(A) \neq 0\}| \leq 1+3C_22^{m+k-n}.
	\end{equation*}
	Combining this with {\bf (P\ref{P4})}, we get
	\begin{align*}
		\|\one_A\|_k^2 &\leq (1+3C_22^{m+k-n})(C_42^{-n-1-k})^2 \\
		&= (C_42^{-n-1-k})^2 + 3C_2C_4^22^{m-3n-k-2},
	\end{align*}
	and then this estimate gives us 
	\begin{align*}
		\sum_{k=n-m-m_2+1}^{n-2} \|\one_A\|_k &< C_42^{-n-1}\sum_{k=n-m-m_2+1}^{\infty}2^{-k} + \sqrt{3C_2}C_4 2^{\frac{m}{2}-\frac{3n}{2}-1}\sum_{k=n-m-m_2+1}^{\infty} 2^{-\frac{k}{2}} \\
		&= C_42^{-n-1}2^{m+m_2-n} + \sqrt{3C_2}C_4 2^{\frac{m}{2}-\frac{3n}{2}-1} 2^{\frac{m+m_2-n}{2}}(\sqrt{2}-1)^{-1} \\
		&= 4^{-n}2^{m-1}C_4(2^{m_2} + \sqrt{3C_2}2^{\frac{m_2}{2}}(\sqrt{2}-1)^{-1}) \\
		&\overset{\eqref{eq:m2def}}{\leq} 4^{-n}2^{m-1}C_4C_2(2+\sqrt{6}(\sqrt{2}-1)^{-1}) \\
		&\overset{\eqref{eq:4^(-n)2^m}}{\leq} 5\|\one_A\|_{W^{1,1}(th)}C_4C_2(2+\sqrt{6}(\sqrt{2}-1)^{-1}) < 40C_4C_2\|\one_A\|_{W^{1,1}(th)}.
	\end{align*}
	Combining with \eqref{eq:smallkestimate}, we conclude that
	\begin{equation*}
		\sum_{k=0}^{n-2} \|\one_A\|_k \leq C_4(20+40C_2)\|\one_A\|_{W^{1,1}(th)}.\qedhere
	\end{equation*}
\end{proof}

\subsection{$L_1$-Distortion of $\tc(Gr_n)$}\label{S:c1Gr}
We now apply the results of the preceding subsections in order to obtain lower bounds on $c_1(\tc(Gr_n))$ in Theorem~\ref{thm:c1(TC(Gr))2}, which proves Theorem~\ref{thm:c1(TC(grids))}.

\begin{theorem} \label{thm:c1(TC(Gr))2}
$c_1(\tc(Gr_n)) \geq 5760^{-1}(n-1)$.
\end{theorem}

\begin{proof}
By Theorems~\ref{thm:c1(TC)}, \ref{thm:(C2)grids}, and \ref{thm:(C1)grids}, we have that
\[c_1(\TC(Gr_n)) \geq (n-1)(C_1C_2C_3C_4)^{-1}(20+40C_2)^{-1},\]
where $C_1,C_2,C_3,C_4$ are constants for which there exist measures $\{\mu_t\}_{t\in T}$ satisfying the properties {\bf(P\ref{P1})}-{\bf(P\ref{P5})}. The term $C_1C_2C_3C_4(20+40C_2)$ can be chosen with an upper bound of $5760$ by Theorem~\ref{thm:Cvalues}.
\end{proof}



Combining Theorem \ref{thm:c1(TC(Gr))2} with the estimate \eqref{E:UpperFRT}, we get as a corollary the first of our two main results.

\begin{theorem1.1} 
$c_1(\tc(\{0,\dots,n\}^2))= \Theta(\log n)$.
\end{theorem1.1}

\section{Edge Replacement and $\os$ products}\label{S:EdgeRepl}

\subsection{Definitions and Basic Concepts}
In the remaining part of this paper, it will be convenient to follow closely the terminology for graphs introduced by Serre \cite[\S2]{Ser03}. Namely, graphs $G$ have finite vertex sets $V(G)$ and directed edge sets $Y(G) \sbs (V(G) \times V(G)) \setminus \{(u,u): u\in V(G)\}$. The set $Y(G)$ contains each edge in both directions, meaning $uv \in Y(G) \iff vu \in Y(G)$. When we need to pick a direction on each edge, we call it {\it a choice of an orientation}, the resulting subset of $Y(G)$ is denoted $E(G)$.

When $e = (u,v) \in Y(G)$, we denote $u$ by $e^+$, $v$ by $e^-$, and $(v,u)$ by $\bar{e}$. A sequence of vertices $\{u_i\}_{i=0}^k \sbs V(G)$ is a {\it directed path} for the orientation $E(G)$ if $(u_{i-1},u_i) \in E(G)$ for every $i\in\{1,\dots k\}$ and is a {\it path} if $(u_{i-1},u_i) \in Y(G)$ for every $i\in\{1,\dots k\}$. A sequence of edges $\{e_i\}_{i=0}^k \sbs E(G)$ is a {\it directed edge path} if $e_{i-1}^+ = e_i^-$ for all $i\in\{1,\dots k\}$ and is an {\it edge path} if $e_i\in\{f_i,\bar{f_i}\}$ for some directed edge path $\{f_i\}_{i=0}^k$.
\medskip

We consider a special class of graphs called $st$-graphs. Classes of graphs similar to the one we consider have been studied under the same name before, but the exact definition can vary from author to author. For us, an {\it $st$-graph} is a  graph $G$, as defined above, having  two distinct vertices, a {\it source} vertex $s_G$ and {\it sink} vertex $t_G$, equipped with a \underline{weighted} {\it geodesic metric} $d_G: V(G)\times V(G) \to [0,1]$, an orientation $E(G)$, and a  {\it thickness function} $th_G: E(G) \to (0,1]$ satisfying the axioms {\bf (a)-(c)} below. Note that the assumption that $th_G$ takes positive values implies that every edge $e\in E(G)$ belongs to the introduced below $\gamma_i$ for at least one $i\in I$.
\begin{enumerate}[{\bf (a)}]
	\item $d_G(s_G,t_G) = 1$,
	\item every directed path in orientation $E(G)$ is a geodesic,
	\item there exists a finite index set $I$ and a collection of (not necessarily distinct) directed edge paths $\{\gamma_i\}_{i\in I}$ from $s_G$ to $t_G$ in $E(G)$ such that $th_G(e) = \frac{|\{i\in I:e\in\gamma_i\}|}{|I|}$ for all $e\in E(G)$.
\end{enumerate}

 We overload notation and write $d_G: Y(G) \to [0,1]$ to denote the {\it length function} $d_G(e) := d_G(e^-,e^+)$. Being a geodesic metric, the values of $d_G$ on $V(G)\times V(G)$ are completely determined by its values on $E(G)$.


For $A\sbs V(G)$, we let $\partial_G A$ denote the {\it edge boundary} of $A$, meaning the set of all edges $e\in E(G)$ for which $|\{e^-,e^+\} \cap A| = 1$. The next proposition contains two fundamental structural results about $st$-graphs.

\begin{proposition} \label{prop:thickness}
	Let $G$ be an $st$-graph. Then the following facts hold.
	\begin{enumerate}[{\rm (1)}]
		\item $\sum_{e\in E(G)}th_G(e)d_G(e) = 1$.
		\item If $A \sbs V(G)$ and $|\{s_G,t_G\} \cap A| = 1$, then $\sum_{e\in\partial_G A} th_G(e) \geq 1$.
	\end{enumerate}
\end{proposition}

\remove{
\cg{{\bf Chris's comment:} It is actually okay to use a weaker version of Lemma \ref{lem:monotonicity} with an inequality instead of equality. All we need to prove the Sobolev inequality is that the perimeter is nondecreasing over the course of the development. I made changes to this effect. I also got rid of the simple connectivity assumption everywhere. It was relevant neither here nor anywhere else in the paper.} }

\begin{proof} For (1), we have
		\begin{align*}
			\sum_{e\in E(G)}th_G(e)d_G(e) = \sum_{e\in E(G)} \frac{|\{i\in I: e\in\gamma_i\}|}{|I|}d_G(e)
			= \frac{1}{|I|}\sum_{i\in I}\sum_{e\in \gamma_i}d_G(e) = 1.
		\end{align*}
		
		For (2), let $A\sbs V(G)$ with $|\{s_G,t_G\}\cap A| = 1$. For each of the directed edge paths $\gamma_i$, there must exist at least one edge $\gamma_i \cap \partial_G A$. Hence, we have
		\begin{align*}
			\sum_{e\in\partial_G A} th_G(e) = \sum_{e\in\partial_G A} \frac{|\{i\in I: e\in\gamma_i\}|}{|I|} = \frac{1}{|I|}\sum_{i\in I}\sum_{e\in\partial_G A \cap \gamma_i} 1 &\geq 1.\qedhere
		\end{align*}
\end{proof}

A fundamental operation on the class of $st$-graphs is edge replacement.

\begin{definition}[Edge Replacement]\label{D:EdgeRep}
	Suppose that $G,H$ are $st$-graphs and $e \in E(G)$. We form the {\it edge replacement} $st$-graph $G \cup_e H$ by replacing $e$ with $H$ in the following way:
	\begin{itemize}
		\item The vertex set is $V(G \cup_e H) := V(G) \sqcup V(H)/\sim$, where $\sim$ is the equivalence relation generated by $e^- \sim s_H$ and $e^+ \sim t_H$. The composition of the inclusion map and the quotient map gives an injection from each of $V(G)$ and $V(H)$ into $V(G \cup_e H)$. We identify $V(G)$ as a subset of $V(G \cup_e H)$ in this way, without making reference to the inclusion and quotient maps. On the other hand, we will use special notation for the injection of $V(H)$ into $V(G \cup_e H)$; the image of $v \in V(H)$ is written as $e\os v$.
		\item With this identification, the source and sink vertices are defined by $s_{G\cup_e H} = s_G$ and $t_{G\cup_e H} = t_G$.
		\item When $e' = (u,v) \in E(H)$, we write $e \os e'$ to denote the ordered pair $(e\os u,e\os v) \in V(G \cup_e H) \times V(G \cup_e H)$ and $e\os E(H)$ to denote $\{e\os e'\}_{e'\in E(H)}$. The directed edge set $E(G \cup_e H)$ is defined to be $(E(G)\setminus\{e\}) \cup (e\os E(H)) \sbs V(G \cup_e H) \times V(G \cup_e H)$. Note that the mapping $E(H) \ni e' \mapsto e\os e' \in e\os E(H)$ is bijective and that the two subsets $E(G)\setminus\{e\}$ and $e\os E(H)$ are disjoint.
		\item The geodesic metric is defined on edges by $d_{G\cup_e H}(e') = d_G(e')$ if $e'\in E(G)\setminus\{e\}$ and $d_{G\cup_e H}(e\os e') = d_G(e)d_H(e')$ if $e'\in E(H)$.
		\item The thickness function is defined by $th_{G\cup_e H}(e') = th_G(e')$ if $e'\in E(G)\setminus\{e\}$ and $th_{G\cup_e H}(e\os e') = th_G(e)th_H(e')$ if $e'\in E(H)$.
	\end{itemize}
\end{definition}

\begin{remark} The informal germ of this construction was implicit in many papers, the earliest we know is \cite{BO79}. Initial steps in its formalization were made in \cite{LR10}.
\end{remark}

We shall use the following proposition without reference throughout this section.

\begin{proposition}
	For all $st$-graphs $G,H$ and edges $e\in E(G)$, the edge replacement graph $G\cup_e H$ satisfies the axioms {\bf (a)-(c)} of an $st$-graph.
\end{proposition}

\begin{proof} Let $G,H$ be $st$-graphs and $e\in E(G)$. The axiom {\bf (a)} (stating that $d_{G\cup_e H}(s_{G\cup_e H},t_{G\cup_e H}) = 1$) and {\bf (b)} (that every directed path in $G\cup_e H$ is a geodesic) are straightforward, and we omit the details. For the axiom {\bf (c)}, we will describe the construction of the directed edge paths $\{\gamma_i\}_{i\in I}$ in $G\cup_e H$.
		
		Let $\{\gamma_i\}_{i\in I_G}$ and $\{\gamma_j\}_{j\in I_H}$ be the directed edge paths in $G$ and $H$, respectively, witnessing the axiom {\bf (c)} of $st$-graphs. The index set for $G\cup_e H$ will be $I_G \times I_H$, and for each $i\in I_G$ and $j\in I_H$, we form the directed edge path $\gamma_{i,j}$ from $s_{G\cup_e H}$ to $t_{G\cup_e H}$ by
		
        \begin{equation*}
			\gamma_{i,j} =\begin{cases} \gamma_i &\hbox{ if }e\notin \gamma_i,\\
            (\gamma_i\setminus\{e\})\cup (e\os \gamma_j)&\hbox{ if }e\in \gamma_i.			    
			\end{cases}
		\end{equation*}
		
        Now let $e' \in E(G\cup_e H)$. There are two cases: (i) $e' \in E(G)\setminus\{e\}$ or (ii) $e' = e \os e''$ for some $e'' \in E(H)$. Assume that we are in case (i). Then we have
		\begin{align*}
			th_{G\cup_e H}(e') &= th_G(e') \\
			&= \frac{|i\in I_G: e'\in\gamma_i|}{|I_G|} \\
			&= \frac{|(i,j)\in I_G\times I_H: e'\in\gamma_i|}{|I_G||I_H|} \\
			&= \frac{|(i,j)\in I_G\times I_H: e'\in\gamma_{i,j}|}{|I_G\times I_H|},
		\end{align*}
		establishing axiom {\bf (c)} in this case. Now assume that case (ii) holds. Then we have
		\begin{align*}
			th_{G\cup_e H}(e') &= th_G(e)th_H(e'') \\
			&= \frac{|i\in I_G: e\in\gamma_i|}{|I_G|}\frac{|j\in I_H: e''\in\gamma_j|}{|I_H|} \\
			&= \frac{|(i,j)\in I_G\times I_H: e\os e''\in\gamma_{i,j}|}{|I_G||I_H|} \\
			&= \frac{|(i,j)\in I_G\times I_H: e'\in\gamma_{i,j}|}{|I_G\times I_H|},
		\end{align*}
		which completes the proof.
\end{proof}

The following facts are the primary geometric properties of edge replacement graphs, and they will be used explicitly or implicitly throughout the section. Their proofs are straightforward and left to the reader. For us, a {\it graph morphism of directed} graphs is a map $f: V(H) \to V(H')$ between vertex sets of graphs such that for all $u,v\in V(H)$, we have $(f(u),f(v)) \in E(H') \iff (u,v) \in E(H)$.

\begin{proposition}[Geometry of Edge Replacement] \label{prop:edgerepgeometry}
	Let $G,H$ be $st$-graphs and $e\in E(G)$. Then the following facts hold.
	\begin{enumerate}[{\rm (1)}]
		\item For any vertices $v \in e\os V(H) \sbs V(G\cup_e H)$ and $u\in V(G) \sbs V(G\cup_e H)$, any path from $u$ to $v$ must have nonempty intersection with $\{e^-,e^+\}$.
		\item The inclusion of $(V(G),d_G)$ into $(V(G\cup_e H),d_{G\cup_E H})$ is an isometric embedding of metric spaces.
		\item The mapping $V(H) \ni v \mapsto e\os v \in V(G\cup_e H)$ is an injective graph morphism for directed graphs.
	\end{enumerate}
\end{proposition}

Next, we define combinations of edge replacements that play a very important role in mathematics and are the main notion for the rest of this paper.

\begin{definition}[$\os$ Products and Restricted $\os$ Products]\label{D:SlashPr}
	Let $G$ and $H$ be two $st$-graphs. Let $E' \sbs E(G)$, and let $E' = \{e_i\}_{i=1}^N$ be any enumeration of $E'$. We form the {\it restricted $\os$ product} $G\os_{E'} H$ of $G$ and $H$ recursively by
	\begin{itemize}
		\item $S^{(0)} := G$,
		\item $S^{(i)} := S^{(i-1)}\cup_{e_{i}} H$ for $1 \leq i \leq N$, and
		\item $G\os_{E'} H := S^{(N)}$.
	\end{itemize}
When $E' = E(G)$, the result is denoted $G\os H$ rather than $G\os_{E(G)} H$. The $st$-graph $G\os H$ is called the {\it $\os$ product} or {\it slash product} of $G$ and $H$.
\end{definition}

Morally, the graph $G\os_{E'} H$ is obtained by replacing every edge in $E'$ with a copy of $H$. The construction is independent of the enumeration of $E'$, up to isomorphism.  We have the following basic properties of $\os$ products.

\begin{proposition}[Basic Properties of $\os$ Products]
	Let $G,H$ be $st$-graphs. Then
	\begin{itemize}
		\item $V(G\os H) = \{e\os v\}_{e\in E(G), v\in V(H)}$ and $|V(G\os H)| = |V(G)| + |E(G)|(|V(H)|-2)$,
		\item $s_{G\os H} = s_G$ and $t_{G\os H} = t_G$,
		\item $E(G\os H) = \{e\os e'\}_{e\in E(G), e'\in E(H)}$ and $|E(G\os H)| = |E(G)||E(H)|$,
		\item $d_{G\os H}(e\os e') = d_G(e)d_H(e')$ for all $e\in E(G)$, $e'\in E(H)$, and
		\item $th_{G\os H}(e\os e') = th_G(e)th_H(e')$ for all $e\in E(G)$, $e'\in E(H)$.
	\end{itemize}
\end{proposition}

\subsection{Elementary Developments}\label{S:ElDevel} Now, we turn our attention to
$st$-graphs that can be obtained by a sequence of iterative edge replacements by ``elementary" $st$-graphs. Specifically, let us say that an $st$-graph $H$ is {\it elementary} if it consists of a single directed path from $s_H$ to $t_H$ with all edges having thickness 1, or it consists of two directed paths from $s_H$ to $t_H$, each containing at least three vertices, that share only their endpoint vertices $\{s_H,t_H\}$ in common and all edges having thickness $\frac{1}{2}$.
Note that the edges of an elementary $st$-graph are permitted to have different lengths. Elementary graphs of the first type described above are called {\it $st$-paths}, and those of the second type are called {\it $st$-cycles}. An $st$-graph is {\it trivial} if it is an $st$-path with only 2 vertices. We say that a sequence of $st$-graphs $(G^{(n)})_{n=0}^{\infty} = (V^{(n)},E^{(n)})_{n=0}^{\infty}$ is an {\it elementary development} if
\begin{itemize}
	\item $G^{(0)}$ is a trivial $st$-path with $E^{(0)} = \{(s,t)\}$ and
	\item for each $n \geq 0$, there exists an edge $e\in E^{(n)}$ and an elementary $st$-graph $H$ (depending on $n,e$) such that $G^{(n+1)} = G^{(n)} \cup_e H$.
\end{itemize}
Note that, although an elementary development formally involves an infinite sequence of graphs, we allow the case that the sequence stabilizes and $G^{(n+1)} = G^{(n)}$ for all sufficiently large $n$, which occurs when the $st$-graph $H$ with $G^{(n+1)} = G^{(n)} \cup_e H$ is trivial. In this way, the definition of elementary development we give can accommodate finite sequences.

Fix an elementary development $(G^{(n)})_{n=0}^{\infty} = (V^{(n)},E^{(n)})_{n=0}^{\infty}$. We write $(\V,d)$ to denote the direct limit of the sequence of metric spaces $(V^{(n)},d_{G^{(n)}})$. That is, $\V$ is the increasing infinite union $\bigcup_{n\geq 0} V^{(n)}$, and $d$ denotes the metric on $\V$ that restricts to $d_{G^{(n)}}$ on each $V^{(n)}$. This metric is well-defined by Proposition~\ref{prop:edgerepgeometry}(2). We emphasize here that $\V$ is a countable set (being the countable union of finite sets), and $(\V,d)$ is generally not complete.

Let $n\geq 0$. If $G^{(n+1)} \neq G^{(n)}$, then the edge $e \in E^{(n)}$ and $st$-graph $H$ for which $G^{(n+1)} = G^{(n)}\cup_e H$ are uniquely determined (up to isomorphism, in the case of $H$) by $G^{(n)}$ and $G^{(n+1)}$ because $e$ is the unique element of $E^{(n)} \setminus E^{(n+1)}$, and $H$ is isomorphic to the subgraph $G_e$ induced by the union of all directed paths in $G^{(n+1)}$ from $e^-$ to $e^+$, with $s_H$ identified with $e^-$, $t_H$ identified with $e^+$, and $d_H$ identified with the restriction of $\frac{d}{d(e)}$ to $G_e$.

Define $\E := \bigcup_{n\geq 0} E^{(n)}$. Note that these unions are in general neither disjoint nor increasing, as, for each $n\geq 0$ with $G^{(n+1)} \neq G^{(n)}$, there is an edge $e \in E^{(n)}$ for which $E^{(n)}\setminus\{e\} \sbs E^{(n+1)}$ and $e\not\in E^{(n+1)}$. Let $th: \E \to (0,1]$ denote the function that restricts to $th_{G^{(n)}}$ on each $E^{(n)}$. This is well defined since if $e \in E^{(n)} \cap E^{(n+1)}$, then $th_{G^{(n+1)}}(e) = th_{G^{(n)}}(e)$. For each $e\in\E$, we let $n_e$ denote the greatest $n \geq 1$ such that $e\in E^{(n-1)}$ (allowing the possibility that $n_e = \infty$). Let $m\geq 0$ with $G^{(m+1)} \neq G^{(m)}$. As previously explained, there is a unique edge $e^{(m)} \in E^{(m)}$ for which there exists an $st$-graph $H^{(m)}$ with $G^{(m+1)} = G^{(m)}\cup_e H^{(m)}$. Let $\prec$ be the strict partial order on $\E$ generated by the relation $e^{(m)}\os e \prec e^{(m)}$ for all such $m \geq 0$ and all $e \in E(H^{(m)})$. We let $\preceq$ denote the corresponding partial order on $\E$ ($e \preceq e'$ if $e \prec e'$ or $e = e'$). Then $(\E,\preceq)$ has the structure of a rooted tree. The {\it root} is $(s,t)$ since it satisfies $(s,t) \succeq e$ for all $e\in\E$, and the tree structure owes to the fact that for each $e \in \E$, the {\it ancestor set} $\E_{\succ e} := \{e'\in\E: e' \succ e\}$ is totally ordered (that is, for any two distinct elements $e'$ and $e''$ of it either $e'\prec e''$ or $e''\prec e$). We say that $e_1,e_2\in\E$ are {\it incomparable} if $e_1 \not\preceq e_2$ and $e_2 \not\preceq e_1$. We turn the reader's attention to the fact that we mean the root is a maximal element in the tree (in many sources, it is required to be minimal).

\begin{lemma}[Bottleneck] \label{lem:bottleneck}
	Suppose $f_1,f_2\in\E$ are incomparable. Suppose that $x_1,x_2\in\V$ are descendant vertices of $f_1,f_2$, respectively, by which we mean vertices of the form $x_1\in \{e_1^-,e_1^+\}$ and $x_2\in \{e_2^-,e_2^+\}$ for some $e_1,e_2\in\E$ with $e_1 \preceq f_1$ and $e_2 \preceq f_2$. Then for every $n\geq 0$ with $x_1,x_2 \in V^{(n)}$, every path in $G^{(n)}$ starting at $x_1$ and ending at $x_2$ must have nonempty intersection first with $\{f_1^-,f_1^+\}$ and then with $\{f_2^-,f_2^+\}$.
\end{lemma}

\begin{proof} To see this, it suffices to use the observation: the only descendant vertices of $f_1$ which are ends of edges $e$ that do not satisfy $e\preceq f_1$ are  $\{f_1^-,f_1^+\}$ (and likewise for $f_2$).
\end{proof}

For $e\in\E$, we define the {\it descendant set} $\E_{\preceq e}$ to be $\{e'\in\E: e' \preceq e\}$. For $n\geq 0$, we define the {\it $n$th descendant edge set} of $e\in\E$ to be the subset $E^{(n)}_{\preceq e} := E^{(n)} \cap \E_{\preceq e}$.

\begin{lemma}[Incomparability Implies Disjointness] \label{lem:incomparabledisjoint}
	If $E' \sbs \E$ is a pairwise incomparable subset, then for every $n\geq 0$, the collection $\{E^{(n)}_{\preceq e}\}_{e\in E'}$ of subsets of $E^{(n)}$ is pairwise disjoint.
\end{lemma}

\begin{proof} It is easy to see, for example, by induction,
that descendants of incomparable edges are incomparable.	
\end{proof}

We need to single out those edges in $\E$ that were replaced by an $st$-cycle at some stage of the development. Specifically, we define $\E_{cyc}$ to be the set of all $e\in\E$ such that there exists $n \geq 0$ an $st$-cycle $H$ with $G^{(n+1)} = G^{(n)}\cup_e H$. It follows immediately from the definitions that $\E_{cyc} \subseteq \{e\in\E: n_e < \infty\}$.

\begin{lemma}[Disjoint Descendants] \label{lem:disjointdes}
	For all $n,k \geq 0$, the collection $\{E^{(n)}_{\preceq e}\}_{\substack{e\in\E_{cyc}\\th(e)=2^{-k}}}$ is pairwise disjoint.
\end{lemma}

\begin{proof}
		Note that, for all $e,e'\in\E_{cyc}$ with $e\neq e'$, if $th(e)=th(e')$, then $e,e'$ are incomparable. This is because, if $e,e'\in\E_{cyc}$ and $e' \prec e$, then, by definition of the thickness function, $th(e') \leq \frac{1}{2}th(e) < th(e)$. Then the conclusion follows from the Incomparability Implies Disjointness Lemma \eqref{lem:incomparabledisjoint}.
\end{proof}

\subsection{Perimeter Measure}
Let $A \sbs \V$ be any subset ($A$ is permitted to be infinite). The {\it edge boundary} is the subset $\partial A = \{e\in\E: |\{e^-,e^+\}\cap A|=1\}$. The {\it perimeter measure} of $A$ is defined on each subset $E' \sbs\E$ by
\begin{equation*}
	\Per_A(E') := \sum_{e\in E'\cap\partial A} th(e) \in [0,\infty].
\end{equation*}

Along with Lemma~\ref{lem:mu(A)<=Per(A)}, the next lemma is one of the two most important results concerning the geometry of developments. These lemmas should be seen as the two basic ingredients going into the Sobolev Inequality for Sets (Lemma~\ref{lem:setSobolev}).

	\begin{lemma}[Monotonicity of Perimeter] \label{lem:monotonicity}
		For all $A\sbs\V$, $e\in\E$, and $n\geq 0$, we have the inequality
		\begin{equation*}
			\Per_A(E^{(n+1)}_{\preceq e}) \geq \Per_A(E^{(n)}_{\preceq e}).
		\end{equation*}
	\end{lemma}
	
	\begin{proof}
		Let $A\sbs\V$, $e\in\E$, and $n \geq 0$. If $E^{(n)}_{\preceq e} = \emptyset$ (which happens if and only if $e$ does not appear until after the $n$th stage of the development), then the inequality is trivially true, and thus we may assume that $E^{(n)}_{\preceq e} \neq \emptyset$.
		
		We will prove that, for all edges $e_{des} \in E^{(n)}_{\preceq e}$
		(where $des$ stands for `descendant'),
		\begin{equation} \label{eq:Per(edes)}
			\Per_A(E^{(n+1)}_{\preceq e_{des}}) \geq \Per_A(\{e_{des}\}).
		\end{equation}
		Before establishing \eqref{eq:Per(edes)}, let us see how it implies the conclusion. Assume that \eqref{eq:Per(edes)} holds for every $e_{des}\in E^{(n)}_{\preceq e}$. Then since the set $E^{(n+1)}_{\preceq e}$ is the disjoint union of the collection of sets $\{E^{(n+1)}_{\preceq e_{des}}\}_{e_{des}\in E^{(n)}_{\preceq e}}$ (this is straightforward from the definitions), we have that 
		\begin{align*}
			\Per_A(E^{(n+1)}_{\preceq e}) &= \sum_{e_{des}\in E^{(n)}_{\preceq e}} \Per_A(E^{(n+1)}_{\preceq e_{des}}) \\
			&\overset{\eqref{eq:Per(edes)}}{\geq} \sum_{e_{des}\in E^{(n)}_{\preceq e}} \Per_A(\{e_{des}\}) \\
			&= \Per_A(E^{(n)}_{\preceq e}),
		\end{align*}
		which is the desired inequality. We now prove \eqref{eq:Per(edes)}.
		
		Let $e_{des} \in E^{(n)}_{\preceq e}$. There are two cases: (a) $e_{des}$ was not replaced at stage $n+1$, meaning $e_{des} \in E^{(n+1)}$, and (b) $e_{des}$ was replaced at stage $n+1$, meaning $e_{des}\not\in E^{(n+1)}$. In case (a), we have that $E^{(n+1)}_{\preceq e_{des}} = E^{(n)}_{\preceq e_{des}}$, and so \eqref{eq:Per(edes)} is satisfied tautologically.
		
		Assume, then, that we are in case (b). Then $G^{(n+1)} = G^{(n)}\cup_{e_{des}} H$ for some $st$-graph $H$, and
		\begin{equation} \label{eq:des(edes)}
			E^{(n+1)}_{\preceq e_{des}} = e_{des}\os E(H)
		\end{equation}
		by definition. There are two subcases to consider: (i) $e_{des}\not\in\partial A$ and (ii) $e_{des}\in\partial A$. In case (i), the right-hand side of \eqref{eq:Per(edes)} is 0, and so the inequality is trivially satisfied.
		
		Now assume that (ii) $e_{des}\in\partial A$. Then, by definition,
		\begin{equation} \label{eq:Per=th}
			\Per_A(\{e_{des}\}) = th(e_{des}).
		\end{equation}
		By Proposition~\ref{prop:edgerepgeometry}(3), the subgraph of $G^{(n+1)}$ induced by $e_{des}\os V(H)$ is naturally identified (as a directed graph) with $H$, and under this identification, $e_{des}^-$ corresponds to $s_H$ and $e_{des}^+$ to $t_H$. Let $A_H$ denote the subset of $V(H)$ identified with $A \cap (e_{des}\os V(H))$. Then $|\{s_H,t_H\} \cap A_H| = 1$ since $e_{des}\in \partial A$. Hence, by Proposition~\ref{prop:thickness}(2),
		\begin{equation} \label{eq:A_H}
			\sum_{e''\in\partial_H A_H} th_H(e'') \geq 1.
		\end{equation}
		Then we have
		\begin{align*}
			\Per_A(E^{(n+1)}_{\preceq e_{des}}) &\overset{\eqref{eq:des(edes)}}{=} \Per_A(e_{des}\os E(H)) \\
			&= \sum_{e'\in (e_{des}\os E(H))\cap \partial A} th(e') \\
			&\overset{\text{Prop }\ref{prop:edgerepgeometry}(3)}{=} \sum_{e''\in\partial_H A_H} th(e_{des}\os e'') \\
			&= \sum_{e''\in\partial_H A_H} th(e_{des})th_H(e'') \\
			&\overset{\eqref{eq:A_H}}{\geq} th(e_{des}) \\
			&\overset{\eqref{eq:Per=th}}{=} \Per_A(\{e_{des}\}),
		\end{align*}
		establishing \eqref{eq:Per(edes)}.
	\end{proof}

	We define for any subset $A\sbs\V$ the {\it total perimeter} of $A$ to be the number
	\begin{equation} \label{eq:Perdef}
		\Per(A) := \sup_{n\geq 0}\Per_A(E^{(n)}) \in [0,\infty].
	\end{equation}


\subsection{A Sobolev Inequality on $st$-Graphs}
\label{ss:st-Sobolev}
We now introduce the measures appearing on the left-hand side of the Sobolev inequality \eqref{E:Sobolev}. Fix $e\in\E_{cyc}$, and let $n\geq 0$ such that $G^{(n+1)} = G^{(n)}\cup_e H$ for some $st$-cycle $H$. Then, by definition, there are two geodesics $\gamma_1,\gamma_2 \sbs V(H)$ from $s_H$ to $t_H$ such that $V(H) = \gamma_1 \cup \gamma_2$ with $\gamma_1 \cap \gamma_2 = \{s_H,t_H\}$. For each $i\in\{1,2\}$, define
\begin{equation*}
	ht_i(e) := \max_{u\in\gamma_i} \dist_{H}(u,\{s_{H},t_{H}\}),
\end{equation*}
choose a vertex $u_i \in \gamma_i$ such that
\begin{equation*}
	\dist_{H}(u_i,\{s_{H},t_{H}\}) = ht_i(e),
\end{equation*}
and then set
\begin{align*}
	x_i(e) &:= e\os u_i \in V^{(n+1)}.
\end{align*}
The definition of an $st$-cycle in \S\ref{S:ElDevel} implies that $ht_i(e)>0$ for $i\in\{1,2\}$.
We define $\mu_e \in \TC(G^{(n+1)})$ to be the signed measure
\begin{equation*}
	\mu_e := th(e)(\delta_{x_1(e)} - \delta_{x_2(e)}).
\end{equation*}

Notice that the only possible values of $|\mu_e(A)|$ are $th(e)$ and 0, and the former value is attained if and only if $|\{x_1(e),x_2(e)\}\cap A| = 1$.

For future use, observe that
\begin{align} \label{eq:d(xi,e)}
	d(e)ht_i(e) &= \dist_H(x_i(e),\{e^-,e^+\}),
\end{align}
and therefore, defining $ht(e) := ht_1(e)+ht_2(e)$, we have
\begin{equation}\label{eq:d(x1,x2)}
	d(x_1(e),x_2(e)) \geq d(e)ht(e).
\end{equation}

Along with Lemma~\ref{lem:monotonicity}, the next lemma is the other basic ingredient going into the Sobolev Inequality for Sets (Lemma~\ref{lem:setSobolev}), by way of the Same-Thickness Sobolev Inequality (Lemma~\ref{lem:same-thicknessSobolev}).

	\begin{lemma} \label{lem:mu(A)<=Per(A)}
		For every $A\sbs\V$, every $e\in\E_{cyc}$, and every $n\geq n_e$, we have the inequality
		\begin{equation*}
			|\mu_e(A)| \leq \Per_A(E^{(n)}_{\preceq e}).
		\end{equation*}
	\end{lemma}
	
	\begin{proof}
		Let $A\sbs\V$ and $e\in\E_{cyc}$. If $\mu_e(A) = 0$, then the inequality holds trivially. So we may assume that $\mu_e(A) \neq 0$, and in this case it must happen that $|\mu_e(A)| = th(e)$. The proof is by induction on $n$. The base case is $n=n_e$. In this case, by definition of $n_e$, there is an $st$-cycle $H$ such that $G^{(n)} = G^{(n-1)}\cup_e H$ and $E^{(n)}_{\preceq e} = e\os E(H)$, and $\mu_e = th(e)(\delta_{x_1(e)}-\delta_{x_2(e)})$ with $|\{x_1(e),x_2(e)\}\cap A| = 1$ (since $\mu_e(A) \neq 0$). The $st$-cycle $e\os V(H)$ is the union of two paths from $x_1(e)$ to $x_2(e)$ whose intersection is $\{x_1(e),x_2(e)\}$. One of the paths goes around one side of the cycle and passes through $e^-$, and the other path goes around the other side of the cycle and passes through $e^+$. Since $|\{x_1(e),x_2(e)\}\cap A| = 1$, each of these edge paths contains at least one edge in $\partial A$, say $e_1$ and $e_2$. By definition, $th(e_1) = th(e_2) = \frac{1}{2}th(e)$. Hence, we have
		\begin{align*}
			|\mu_e(A)| = th(e)
			= th(e_1) + th(e_2)
			\leq \Per_A(e\os E(H))
			= \Per_A(E^{(n)}_{\preceq e}).
		\end{align*}
		The inductive step follows from Monotonicity of Perimeter (Lemma~\ref{lem:monotonicity}).
	\end{proof}

\begin{lemma}[Same-Thickness Sobolev Inequality] \label{lem:same-thicknessSobolev}
	For all $k\geq 0$ and all $A\sbs\V$,
	\begin{equation*}
		\sum_{\substack{e\in\E_{cyc}\\th(e)=2^{-k}}} |\mu_e(A)| \leq \Per(A).
	\end{equation*}
\end{lemma}

\begin{proof}
	Follows from the Disjoint Descendants Lemma \ref{lem:disjointdes}, Lemma~\ref{lem:mu(A)<=Per(A)}, and \eqref{eq:Perdef}.
\end{proof}

\begin{lemma}[Sobolev Inequality for Sets] \label{lem:setSobolev}
	For all $A\sbs\V$,
	\begin{equation*}
		\sum_{k\geq 0} \left(\sum_{\substack{e\in\E_{cyc}\\th(e)=2^{-k}}} |\mu_e(A)|^2\right)^{\frac12} \leq (1-2^{-\frac12})^{-1}\Per(A).
	\end{equation*}
\end{lemma}

\begin{proof}
	Let $A\sbs\V$. Let $k_*$ be the minimal value of $k$ such that there exists $e\in\E_{cyc}$ with $th(e) = 2^{-k}$ and $\mu_e(A) \neq 0$. It is easy to see that for any $e\in\E_{cyc}$, the only possible values of $|\mu_e(A)|$ are $th(e)$ and 0. Therefore, we must have by the Same-Thickness Sobolev Inequality (Lemma~\ref{lem:same-thicknessSobolev}) that
	\begin{equation} \label{eq:Per(A)k*}
		\Per(A) \geq 2^{-k_*}.
	\end{equation}
	Then we have
	\begin{align*}
		\sum_{k\geq 0} \left(\sum_{\substack{e\in\E_{cyc}\\th(e)=2^{-k}}} |\mu_e(A)|^2\right)^{\frac12} &= \sum_{k\geq k_*} \left(\sum_{\substack{e\in\E_{cyc}\\th(e)=2^{-k}}} |\mu_e(A)|^2\right)^{\frac12} \\
		&\leq \sum_{k\geq k_*} \left(\max_{\substack{e\in\E_{cyc}\\th(e)=2^{-k}}} |\mu_e(A)|\sum_{\substack{e\in\E_{cyc}\\th(e)=2^{-k}}} |\mu_e(A)|\right)^{\frac12} \\
		&\leq \sum_{k\geq k_*} 2^{-k/2}\left(\sum_{\substack{e\in\E_{cyc}\\th(e)=2^{-k}}} |\mu_e(A)|\right)^{\frac12} \\
		&\overset{\text{Lem }\ref{lem:same-thicknessSobolev}}{\leq} \sum_{k\geq k_*} 2^{-k/2}\Per(A)^{\frac12} \\
		&= (1-2^{-\frac12})^{-1}2^{-k_*/2}\Per(A)^{\frac12} \\
		&\overset{\eqref{eq:Per(A)k*}}{\leq} (1-2^{-\frac12})^{-1}\Per(A).
	\end{align*}
\end{proof}

For $n\geq 0$ and $f: V^{(n)} \to \R$, we define $\nabla f: E^{(n)} \to \R$ by
\begin{equation*}
	\nabla f(e) = \frac{f(e^+)-f(e^-)}{d(e)}.
\end{equation*}

\begin{theorem}[Sobolev Inequality] \label{thm:Sobolev}
	Let $E'\sbs\E_{cyc}$ be a finite subset and $n\geq \max_{e\in E'}n_e$. Then for all $f: V^{(n)} \to \R$,
	\begin{equation*}
		\sum_{k\geq 0} \left(\sum_{\substack{e\in E'\\th(e)=2^{-k}}} \Big|\int f d\mu_e\Big|^2\right)^{\frac12} \leq (1-2^{-\frac12})^{-1}\sum_{e\in E^{(n)}} th(e)d(e)|\nabla f(e)|.
	\end{equation*}
\end{theorem}

\begin{proof}
	Let $A \sbs V^{(n)}$. By Theorem \ref{thm:iso-to-Sobolev}, it suffices to prove the theorem assuming that $f = \one_A$. 
	Obviously, all quantities in the required inequality depend only on the graphs $G^{(m)}$ for $m \leq n$. Thus, we may assume that the development $(G^{(m)})_{m \geq 0}$ stabilizes after the $n$th stage, meaning $G^{(m)} = G^{(n)}$ for all $m \geq n$. In this case, Monotonicity of Perimeter (Lemma~\ref{lem:monotonicity}) and \eqref{eq:Perdef} imply
	\begin{equation*}
		\Per(A) = \Per_A(E^{(n)}),
	\end{equation*}
	and it is easy to see that
	\begin{equation*}
		\Per_A(E^{(n)}) = \sum_{e\in E^{(n)}}th(e)d(e)|\nabla \one_A(e)|.
	\end{equation*}
	With these two observations in mind, we see that the conclusion of the theorem is precisely the Sobolev Inequality for Sets (Lemma~\ref{lem:setSobolev}).
\end{proof}

\subsection{Transportation Cost Norm of Signed Sums}


\begin{theorem} \label{thm:TClowerbound}
Let $k\geq 0$, $E'$ be a finite subset of $\E_{cyc}$ with $th(e) = 2^{-k}$ for all $e\in E'$, and $n\geq \max_{e\in E'}n_e$. Then for all choices of signs $\{\eps_e\}_{e\in E'} \sbs \{-1,1\}$, we have that
	\begin{equation*}\label{E:InGraphPar}
		\Big\|\sum_{e\in E'}\eps_e\mu_e\Big\|_{\TC(G^{(n)})} \geq \sum_{e\in E'} th(e)d(e)ht(e).
	\end{equation*}
\end{theorem}


\begin{proof}
	Let $\{\eps_e\}_{e\in E'} \sbs \{-1,1\}$. Define $f: \{x_1(e),x_2(e)\}_{e\in E'} \to \R$ by
	\begin{align*}
		f(x_1(e)) &:= \eps_e d(e)ht_1(e), \\
		f(x_2(e)) &:= -\eps_e d(e)ht_2(e).
	\end{align*}
	We will see that $f$ is 1-Lipschitz. Let $e,e'\in E'$ with $e\neq e'$. Let $i,j\in\{1,2\}$. By the Same-Thickness Incomparability (Lemma~\ref{lem:same-thicknessSobolev}), $e,e'$ are incomparable. Then by the Bottleneck Lemma \eqref{lem:bottleneck}, every path in $G^{(n)}$ starting at $x_i(e)$ and ending at $x_j(e')$ must have nonempty intersection first with $\{e^-,e^+\}$ and then with $\{(e')^-,(e')^+\}$. This implies
	\begin{align*}
		d(x_i(e),x_j(e')) &\geq \dist(x_i(e),\{e^-,e^+\}) + \dist(x_j(e'),\{(e')^-,(e')^+\}) \\
		&\overset{\eqref{eq:d(xi,e)}}{=} d(e)ht_i(e) + d(e')ht_j(e') \\
		&= |f(x_i(e))| + |f(x_j(e'))| \\
		&\geq |f(x_i(e))-f(x_j(e'))|.
	\end{align*}
	To conclude that $f$ is 1-Lipschitz, it remains to observe that $|f(x_1(e))-f(x_2(e))| \leq d(x_1(e),x_2(e))$ due to \eqref{eq:d(x1,x2)}. We extend $f$ as $1$-Lipschitz to $V^{(n)}$ using the well-known McShane-Whitney extension theorem. The choice of the extension is not essential because $f$ will be never applied to points outside   $\{x_1(e),x_2(e)\}_{e\in E'}$.
    
    Using Kantorovich duality, we finish the proof with the estimate
	\begin{align*}
		\Big\|\sum_{e\in E'}\eps_e\mu_e\Big\|_{\TC(G^{(n)})} &\geq \sum_{e\in E'}\eps_e \int f d\mu_e \\
		&= \sum_{e\in E'}\eps_e (\eps_e d(e)ht_1(e)th(e) + \eps_e d(e)ht_2(e)th(e)) \\
		&= \sum_{e\in E'}th(e)d(e)(ht_1(e)+ht_2(e)) \\
		&= \sum_{e\in E'}th(e)d(e)ht(e). \qedhere
	\end{align*}
\end{proof}

\subsection{The $L_1$-Distortion of $\TC(G^{(n)})$}

\begin{theorem}[$L_1$-Distortion] \label{thm:c1(TC(dev))}
	Let $E'\sbs\E_{cyc}$ be a finite subset and $n\geq \max_{e\in E'}n_e$. Then
	\begin{equation*}
		c_1(\TC(G^{(n)})) \geq \frac{2-\sqrt{2}}{4}\sum_{e\in E'} th(e)d(e)ht(e).
	\end{equation*}
   \end{theorem}

\begin{proof}
As in the proof for planar grids, we will verify that {\bf(C\ref{C1})}-{\bf(C\ref{C2})} are satisfied for suitable measures and then apply Theorem~\ref{thm:c1(TC)}. For $k \in \{0,\dots n-1\}$,  let $E'(k)$ be the subset of $E'$ consisting of those $e\in E'$ with $th(e)=2^{-k}$. For $k\in\{0,\dots n-1\}$ such that $E'(k) \neq \emptyset$, choose a subset $E''(k) \sbs E'(k)$ such that
\begin{itemize}
    \item[(i)] $|E''(k)|$ is a power of 2 and
    \item[(ii)] $\sum_{e\in E''(k)} th(e)d(e)ht(e) \geq \frac{1}{2}\sum_{e\in E'(k)} th(e)d(e)ht(e)$.
\end{itemize}
We will verify {\bf(C\ref{C1})}-{\bf(C\ref{C2})} for the measures $\{\{\mu_e\}_{e\in E''(k)}\}_{k=0}^{n-1}$.

Since $|E''(k)|$ is a power of 2, we can find an orthogonal system of functions $\{\theta_i: E''(k) \to \{-1,1\}\}_{i\in E''(k)}$ by Sylvester's construction of Hadamard matrices \cite[$\S$~2.1.1]{Hor07} (also known as Walsh matrices). By Theorem~\ref{thm:TClowerbound} and item (ii) above, this system satisfies inequality \eqref{E:TCtheta} of condition {\bf(C\ref{C2})} with $\alpha_k = \frac{1}{2}\sum_{e\in E'(k)} th(e)d(e)ht(e)$. By Proposition~\ref{prop:thickness}(1), equation \eqref{eq:probabilitymeasure} of condition {\bf(C\ref{C1})} is satisfied, and by Theorem~\ref{thm:Sobolev}, inequality \eqref{E:Sobolev} of condition {\bf(C\ref{C1})} is satisfied with constant $C \leq (1-2^{-\frac12})^{-1}$. Hence, by Theorem~\ref{thm:c1(TC)}, we have that
\begin{equation*}
    c_1(\TC(G^{(n)})) \geq (1-2^{-\frac12}) \sum_{k=0}^{n-1} \tfrac{1}{2}\sum_{e\in E'(k)} th(e)d(e)ht(e) = \frac{2-\sqrt{2}}{4} \sum_{e\in E'} th(e)d(e)ht(e). \qedhere
\end{equation*}
\end{proof}

\subsection{Slash Powers of Cycles-with-Handles}
For $n\geq 0$, the {\it slash power} $G^{\os n}$ is defined recursively by
\begin{itemize}
	\item $G^{\os 0}$ is a trivial $st$-path.
	\item $G^{\os n+1} = G^{\os n}\os G$.
\end{itemize}
It is easy to see by induction that
\begin{itemize}
	\item $|E(G^{\os n})| = |E(G)|^n$ and
	\item $|V(G^{\os n})| = 2+(|V(G)|-2)\frac{|E(G)|^n-1}{|E(G)|-1}$. 
\end{itemize}
Notice that there is a constant $C<\infty$ (depending on $G$ but not $n$) such that $|V(G^{\os n})| \leq C^n$ for every $n\geq 0$.

Whenever $Pa$ is an $st$-path, $e_{rep}\in E(Pa)$ is an edge that gets replaced, and $Cy$ is an $st$-cycle, we call the $st$-graph $H = Pa\cup_{e_{rep}} Cy$ a {\it cycle-with-handles}. For the remainder of this section, fix a cycle-with-handles $H = Pa\cup_{e_{rep}} Cy$. Let $G$ be an $st$-graph. We call $G\os Pa$ the {\it subdivision} of $G$ and denote it by $G_{sub}$. Observe that, if we define $E_{rep}(G_{sub}) := \{e\os e_{rep}\}_{e\in E(G)} \sbs E(G_{sub})$, then $G\os H$ is equal to the restricted $\os$ product $G_{sub}\os_{E_{rep}(G_{sub})} Cy$. Henceforth, we will suppress notation and write $G_{sub}\os_{rep} Cy$ instead of $G_{sub}\os_{E_{rep}(G_{sub})} Cy$. Obviously, for every $n\geq 0$, we have that $G^{\os n+1} = (G^{\os n})_{sub} \os_{rep} Cy$.

\begin{example}[Diamond Graphs]\label{E:Diam}
	The first {\it diamond graph} $D = Pa_D\cup_{e_{rep}} Cy_D$ is a cycle-with-handles where the $st$-path $Pa_D$ is trivial and the $st$-cycle $Cy_D$ is a union of two directed $st$-paths, each containing two edges of length $\frac{1}{2}$. The {\it diamond graphs} are the graphs in the sequence of $\os$ powers $\{D^{\os n}\}_{n\geq 0}$.
\end{example}

\begin{example}[Laakso Graphs]\label{E:Laakso}
	The first {\it Laakso graph} $La = Pa_{La}\cup_{e_{rep}}Cy_{La}$ is a cycle-with-handles where the $st$-path $Pa_{La}$ contains three edges -- the two outer edges -- one of which contains $s_{La}$ and the other of which contains $t_{La}$ -- each has length $\frac{1}{4}$, and the middle edge is $e_{rep}$ with length $\frac{1}{2}$. The $st$-cycle $Cy_{La}$ equals $Cy_D$. The {\it Laakso graphs} are the graphs in the sequences of $\os$ powers $\{La^{\os n}\}_{n\geq 0}$.
\end{example}

Let $\gamma_1,\gamma_2 \sbs V(Cy)$ be directed paths from $s_{Cy}$ to $t_{Cy}$ with $\gamma_1 \cap \gamma_2 = \{s_{Cy},t_{Cy}\}$. Define the parameters
\begin{align*}
	d_{rep} &:= d_{Pa}(e_{rep}), \\
	ht_{cyc} &:= \max_{u_1\in\gamma_1}\dist_{Cy}(u_1,\{s_{Cy},t_{Cy}\}) + \max_{u_2\in\gamma_2}\dist_{Cy}(u_2,\{s_{Cy},t_{Cy}\}).
\end{align*}

\begin{theorem}[$L_1$-Distortion of $\TC$(Cycle-with-Handles)] \label{thm:c1(TC(cwh))}
	Let $H = Pa\cup_{e_{rep}} Cy$ be a cycle-with-handles. Using the notation of the previous paragraph, we have that
	\begin{equation*}
		c_1(\TC(H^{\os n})) \geq \frac{2-\sqrt{2}}{4}d_{rep}ht_{cyc} \cdot n
	\end{equation*}
	for every $n\geq 0$.
\end{theorem}

\begin{proof}
	Since restricted $\os$ products are, by definition, constructed via iterated edge replacement, and since cycle with handles are, by definition, constructed via edge replacement of elementary graphs, there is an elementary development $(G^{(n)})_{n\geq 0}$ and a pair of increasing, interlacing sequences of integers $0 = n_0 < n'_0 < n_1 < n'_1 < n_2 < n'_2 <\dots$ such that, for all $m \geq 0$,
	\begin{itemize}
		\item $H^{\os m} = G^{(n_m)}$ and
		\item $(H^{\os m})_{sub} = G^{(n'_{m+1})}$.
	\end{itemize}
	We use the notation $\E,\E_{cyc}$ and $th: \E \to [0,1]$, $d: \E \to [0,1]$, and $ht:\E_{cyc} \to [0,1]$ for this development, as previously defined in this section. Recall that $E_{rep}((H^{\os m})_{sub}) = \{e\os e_{rep}\}_{e\in E(H^{\os m})} \sbs E(H^{\os m}\os Pa)$. The following facts hold immediately from the definitions:
	\begin{itemize}
		\item $ht(e) = ht_{cyc}$ for all $e\in\E_{cyc}$.
		\item $E_{rep}((H^{\os m})_{sub}) \sbs \E_{cyc}$ for every $m\geq 0$.
		\item $E_{rep}((H^{\os m_1})_{sub}) \cap E_{rep}((H^{\os m_2})_{sub}) = \emptyset$ whenever $m_1\neq m_2$.
		\item For all $n-1\geq m\geq 0$ and $e'\in E_{rep}((H^{\os m})_{sub})$, it holds that $\{x_1(e'),x_2(e')\} \sbs V(H^{\os n})$.
	\end{itemize}
	Therefore, for any $n\geq 0$, these four items and Theorem~\ref{thm:c1(TC(dev))} imply the lower bound
	\begin{align*}
		c_1(\TC(H^{\os n})) &\overset{\text{Thm }\ref{thm:c1(TC(dev))}}{\geq} \frac{2-\sqrt{2}}{4}\sum_{m=0}^{n-1}\sum_{e'\in E_{rep}((H^{\os m})_{sub})} th(e')d(e')ht_{cyc} \\
		&= \frac{2-\sqrt{2}}{4}\sum_{m=0}^{n-1}\sum_{e\in E(H^{\os m})} th(e\os e_{rep})d(e\os e_{rep})ht_{cyc} \\
		&= \frac{2-\sqrt{2}}{4}\sum_{m=0}^{n-1}\sum_{e\in E(H^{\os m})} th(e)th_{Pa}(e_{rep})d(e)d_{Pa}(e_{rep})ht_{cyc} \\
		&= \frac{2-\sqrt{2}}{4}d_{rep}ht_{cyc}\sum_{m=0}^{n-1}\sum_{e\in E(H^{\os m})} th(e)d(e) \\
		&\overset{\text{Prop }\ref{prop:thickness}(1)}{=} \frac{2-\sqrt{2}}{4}d_{rep}ht_{cyc}\sum_{m=0}^{n-1} 1 \\
		&= \frac{2-\sqrt{2}}{4}d_{rep}ht_{cyc}\cdot n.
	\end{align*}
\end{proof}

For the diamond graph $D$, we have that $d_{rep} = ht_{cyc} = 1$, and for the Laakso graph $La$, we have that $d_{rep} = \frac{1}{2}$ and $ht_{cyc} = 1$. Therefore, we obtain the following corollary of Theorem~\ref{thm:c1(TC(cwh))}.

\begin{corollary}[$L_1$-Distortion of $\TC$(diamond graphs) and $\TC$(Laakso graphs)]\label{C:DiamLaak}
	For every $n\geq 0$,
	\begin{align*}
		c_1(\TC(D^{\os n})) &\geq \frac{2-\sqrt{2}}{4}n, \\
		c_1(\TC(La^{\os n})) &\geq \frac{2-\sqrt{2}}{8}n.
	\end{align*}
\end{corollary}

Now we turn to our second main result:

\begin{theorem1.2}[$L_1$-Distortion of $\TC(G^{\os n})$]\label{T:1.2Rep}
	Let $G$ be any $st$-graph that is not an $st$-path and has at least three vertices. Then there exists a constant $C<\infty$ (depending on $G$ but not $n$), such that $c_1(\TC(G^{\os n})) \geq C^{-1}\log |V(G^{\os n})|$ for every $n\geq 0$.    
\end{theorem1.2}

\begin{proof} This result is an immediate consequence of Theorem \ref{thm:c1(TC(cwh))} if $G$ contains isometrically an $st$-graph $H$ which is a cycle-with-handles, because in this case the containment $V(H^{\os n}) \sbs V(G^{\os n})$ is isometric for every $n\geq 0$. One can find such a subgraph $H$ by choosing a shortest cycle in $G$ and attaching two handles to it. It does not solve the problem, only if one ``side'' of the cycle is an edge. In such a case, we consider $G^{\os 2}$. It contains a subdivision of the subgraph $H$ that embeds into $G^{\os 2}$ isometrically.
\end{proof}

\section{Reduction to Linear Maps} \label{sec:NSred}

The goal of this section is to describe a proof, for a finite metric space $X$, of the equality \eqref{E:NSred} stated in the Introduction \S\ref{sec:intro}. Recall that \eqref{E:NSred} states
\[c_1(\EMD(X)) = c_1(\tc(X)) = c_{1,{\rm lin}}(\tc(X)),
\]
where $c_{1,{\rm lin}}(\tc(X))$ denotes the infimal distortion among all $m\in\N$ and all {\it linear} embeddings $f: \tc(X) \to \ell_1^m$.

The proof essentially belongs to Naor and Schechtman \cite[Lem\-ma~3.1]{NS07}, although they stated it only for a special case.

First, we note that the inequalities 
\[c_1(\EMD(X)) \le c_1(\tc(X)) \le c_{1,{\rm lin}}(\tc(X))
\]
are immediate. It suffices only to prove that for each $\ep>0$ we have
\[c_{1,{\rm lin}}(\tc(X))\le c_1(\EMD(X))+\eps
\]

We denote the unit ball of a Banach space $V$ by $B_V$.

Let $\eps>0$ and let $F:\EMD(X)\to L_1$ be a non-contractive embedding into some $L_1$-space satisfying $||F(\mu)-F(\nu)||_{L_1}\le D\EMD(\mu,\nu)$, where $D<c_1(\EMD(X))+\frac{\eps}{3}$. We use this embedding to construct a linear embedding of 
$B_{\tc(X)}$ into the $L_1$-space with the same distortion. 

Denote by $r$ the minimal distance between the elements of $X$ and let $n=|X|$. Then the formula
\begin{equation} \label{eq:EMD-TC}
   g(\mu)=\frac{r}n\mu+u_X 
\end{equation}
defines a map $g:B_{\tc(X)}\to\EMD(X)$, where $u_X$ is the uniform probability measure on $X$. Note that no measure $\mu \in \tc(X)$ with norm $\leq 1$ can assign to any point mass more than $r^{-1}$, because the mass at that point must be transported a distance at least $r$ in any transportation plan. On the other hand, $u_X$ assigns mass $\frac1n$ to every point. Therefore, $g(\mu)$ is indeed a probability measure on $X$.

Furthermore, it is easy to see that
\[\EMD(g(\mu),g(\nu))=\frac rn\|\mu-\nu\|_\tc.\]
Therefore, $H:B_{\tc(X)}\to L_1$ defined as $H = \frac nr\cdot F\circ g$ is a non-contractive map with $\|H\tau-H\sigma\|_{L_1}\le D\|\tau-\sigma\|_\tc$.

We now utilize the following special case of Bourgain's Discretization Theorem (see \cite{Bou87}, \cite{GNS12}, and \cite[Chapter~9]{Ost13} for the full statement and proof of Bourgain's Discretization Theorem).

\begin{theorem}[{\cite[Theorem~1.3]{GNS12}}] \label{thm:discretization}
	For every $\beta>0$ and every finite-dimensional Banach space $V$, there exists $\delta>0$ such that, for all $\delta$-nets $N_\delta$ in $B_V$, there exists a linear isomorphic embedding of $V$ into an $L_1$-space with distortion at most $c_1(N_\delta)+\beta$.
\end{theorem}

The estimates for the mapping $H$ imply that $c_1(N_\delta)\leq D$ for each $\delta>0$ and $\delta$-net $N_\delta \sbs B_{\tc(X)}$. Therefore, applying Theorem~\ref{thm:discretization}, we get that $\tc(X)$ linearly isomorphically embeds into an $L_1$-space with distortion $\leq D+\frac{\eps}{3}$. Then \cite[Theorem~F.2(i)]{BL00} implies that there is some $m\in\N$ such that $\tc(X)$ linearly isomorphically embeds into $\ell_1^m$ with distortion $\leq D+\frac{2\eps}{3}$. Hence, we conclude $c_{1,{\rm lin}}(\tc(X)) \leq D+\frac{2\eps}{3} < c_1(\EMD(X)) + \eps$, completing the proof. \\

Recall that $M_2^s$ is the subset of $\EMD(\R^2)$ consisting of uniform distributions on $s$-point subsets of $\R^2$.
\begin{theorem}\label{T:EMDuniformR2}
$c_1(M_2^s) = \Omega(\log s)$.
\end{theorem}

\begin{proof}
Let $G_n$ denote the $n \times n$ grid graph, with unit edge lengths. We think of $G_n$ as a subset of $\Z^2 \sbs \R^2$ with the $\ell_1$-metric. Let $\tilde{M}_2^s$ denote the subset of $\EMD(G_n) \sbs \EMD(\R^2)$ consisting of probability measures $\mu$ such that $\mu(\{x\}) \in (\frac{1}{s}\Z) \cap [0,1]$ for all $x\in G_n$. It is easy to see that the closure of $M_2^s$ in $\EMD(\R^2)$ contains $\tilde{M}_2^s$, and hence that $c_1(M_2^s) \geq c_1(\tilde{M}_2^s)$. Indeed, let $\mu \in \tilde{M}_2^s$ and let $\eps>0$. Suppose that $\mu$ is supported on the points $x_1 \dots x_k \in G_n$. For each $x_i \in \{x_1\dots x_k\}$, let $m_i \in \Z \cap [1,s]$ such that $\mu(\{x_i\}) = \frac{m_i}{s}$. For each $x_i \in \{x_1\dots x_k\}$, choose ``satellite" points $y_i^1, \dots y_i^{m_i} \in \R^2$ of $x_i$ such that
\begin{itemize}
    \item the collection $\{y_i^j\}_{i=1,j=1}^{k,m_i}$ is distinct and
    \item $\|x_i - y_i^j\| < \eps$ for all $i \leq k$ and all $j \leq m_i$.
\end{itemize}
Then it is easy to see that the cardinality of the set $\{y_i^j\}_{i=1,j=1}^{k,m_i}$ is $s$ and that the uniform distribution on this set has EMD-distance to $\mu$ at most $\eps$. 

It is also easy to see that $\tilde{M}_2^s$ is $\frac{2n^3}{s}$-dense in $\EMD(G_n)$. Indeed, given $\mu \in \EMD(G_n)$, obtain $\mu_1$ by rounding up each value $\mu(\{x\})$ to the nearest number $\mu_1(\{x\})$ in $(\frac{1}{s}\Z) \cap [0,1]$. The value of $\mu$ at any of the $n^2$ points changes by at most $\frac{1}{s}$. However, it is generally the case that we rounded up too much so that $\mu_1$ is not a probability measure. We know that $\mu_1(G_n) \in (\frac{1}{s}\Z) \cap [1,1+\frac{n^2}{s}]$, so that the remainder $R = \mu_1(G_n) - 1$ belongs to  $(\frac{1}{s}\Z) \cap [0,\frac{n^2}{s}]$, which implies $sR \in \Z \cap [0,n^2]$. Choose any collection of $sR$ points $x_1 \dots x_{sR} \in G_n$, and define $\mu_2$ by $\mu_2(\{x\}) = \mu_1(\{x\}) - \frac{1}{s}$ for $x \in \{x_1,\dots x_{sR}\}$ and $\mu_2(\{x\}) = \mu_1(\{x\})$ for $x \not\in \{x_1,\dots x_{sR}\}$. Then we have that
\begin{itemize}
    \item $\mu_2$ takes values in $(\frac1s\Z) \cap [0,1]$,
    \item $\mu_2$ is a probability measure, and
    \item $|\mu(\{x\}) - \mu_2(\{x\})| \leq \frac{1}{s}$ for any $x\in G_n$.
\end{itemize}
These imply that $\mu_2 \in \tilde{M}_2^s$ and $\|\mu-\mu_2\|_{\rm TV} \leq \frac{n^2}{s}$. Then we use the estimate $\|\mu-\mu_2\|_{\TC} \leq \diam(G_n)\|\mu-\mu_s\|_{\rm TV} \leq 2n \frac{n^2}{s}$ to conclude $\frac{2n^3}{s}$-density.

Let $N_\delta$ be a maximal $\delta$-separated subset of $B_{\TC(G_n)}$, with $\delta>0$ to be chosen later. By the formula \eqref{eq:EMD-TC}, there exists a scaled isometric embedding $g: B_{\TC(G_n)} \hookrightarrow \EMD(G_n)$ with scaling factor $\frac{1}{n^2}$. Hence, $g(N_\delta)$ is a maximal $\frac{\delta}{n^2}$-separated subset of $g(B_{\TC(G_n)})$. Let $\pi: g(N_\delta) \to \tilde{M}_2^s$ be a nearest neighbor projection. Since $\tilde{M}_2^s$ is $\frac{2n^3}{s}$-dense in $\EMD(G_n)$ and since $g(N_\delta)$ is $\frac{\delta}{n^2}$-separated, if we assume that $\frac{2n^3}{s} \leq \frac{1}{3}\frac{\delta}{n^2}$, then $\pi$ will be a 6-biLipschitz embedding. Towards this end, we take $\delta = \frac{6n^5}{s}$ and consequently get $c_1(\tilde{M}_2^s) \geq \frac{1}{6}c_1(g(N_\delta)) = \frac{1}{6}c_1(N_\delta)$. By \cite[Corollary~1.4]{GNS12}, if $\frac{6n^5}{s} = \delta = O(\frac{1}{n^5})$, then we have that $c_1(N_\delta) \geq \frac12c_1(\TC(G_n))$, which equals $\Omega(\log n)$ by Theorem~\ref{thm:c1(TC(grids))}. Put another way, if $s = \Omega(n^{10})$, then $c_1(N_\delta) = \Omega(\log n)$. In summary, if $s = \Omega(n^{10})$, then
$$c_1(M_2^s) \geq c_1(\tilde{M}_2^s) \geq \frac{1}{6}c_1(g(N_\delta)) = \frac{1}{6}c_1(N_\delta) = \Omega(\log n).$$

This implies $c_1(M_2^s) = \Omega(\log s)$.
\end{proof}

\section{Acknowledgments}
This collaboration began during the Workshop in Analysis and Probability at Texas A\&M University in July 2024. We would like to thank the organizers of the workshop and the support of the National Science Foundation under the grant DMS-1900844.


\begin{thebibliography}{GNRS04}
	
	\begin{small}
	
	\bibitem[AIK08]{AIK08}
A.~Andoni, P.~Indyk, R.~Krauthgamer.
\newblock Earth mover distance over high-dimensional spaces. 
\newblock In {\it Proc. of the 19th Ann. ACM-SIAM Symp, on Discrete Algorithms}, pages 343--352, 2008.

\bibitem[BI14]{BI14} A.~Ba\v ckurs, P.~Indyk, Better embeddings for planar earth-mover distance over sparse sets. {\it Computational geometry} (SoCG'14), 280–289, ACM, New York, 2014.

\bibitem[Bar96]{Bar96} Y. Bartal, Probabilistic approximation of metric spaces and its algorithmic applications,
In: {\it 37th Annual Symposium on Foundations of Computer Science}
(Burlington, VT, 1996),  184--193, IEEE Comput. Soc. Press, Los
Alamitos, CA, 1996.

\bibitem[BGS23]{BGS23} 
F.~Baudier, C.~Gartland, Th.~Schlumprecht. 
\newblock $L_1$-distortion of Wasserstein metrics into $L_1$: A tale of two dimensions. 
\newblock {\it Trans. Amer. Math. Soc.} Ser. B {\bf 10} (2023), 1077--1118.

\bibitem[BL00]{BL00} Y.~Benyamini, J.~Lindenstrauss, {\it Geometric nonlinear
	functional analysis}. Vol. {\bf 1}. American Mathematical Society
Colloquium Publications, {\bf 48}. American Mathematical Society,
Providence, RI,
2000.

\bibitem[BO79]{BO79} A.\,N. Berker, S. Ostlund, \emph{Renormalisation-group calculations of finite systems: order parameter and specific heat for epitaxial ordering,} Journal of Physics C: Solid State Physics, \textbf{ 12} (1979), 4961--4976.

\bibitem[BM08]{BM08} J.\,A.~Bondy, U.\,S.\,R.~Murty,  {\it Graph theory.} Graduate Texts in Mathematics, {\bf 244}. Springer, New York, 2008.

\bibitem[Bou87]{Bou87} 
J.~Bourgain.  
\newblock Remarks on the extension of Lipschitz maps defined on discrete sets and uniform homeomorphisms. 
\newblock In: {\it Geometrical aspects of functional analysis} (1985/86), 157--167, {\it Lecture Notes in Math.}, 
{\bf 1267}, Springer, Berlin, 1987.

\bibitem[BC05]{BC05} B.~Brinkman, M.~Charikar,
On the impossibility of dimension reduction in $\ell\sb 1$, {\it
J. ACM}, {\bf 52} (2005), no. 5, 766--788.

\bibitem[Cha02]{Cha02} 
M.~Charikar. 
\newblock Similarity estimation techniques from rounding algorithms. 
\newblock In {\it Proc. of the 34th Ann. ACM Symp. on Theory of Computing}, pages 380--388, 2002.

\bibitem[CJLW22]{CJLW22}
X.~Chen, R.~Jayaram, A.~Levi, E.~Waingarten. 
\newblock New streaming algorithms for high dimensional EMD and MST. 
\newblock In {\it Proc. of the 54th Ann. ACM Symp. on Theory of Computing}, pages 222--233, 2022.

\bibitem[CK13]{CK13} J.~Cheeger, B.~Kleiner, Realization of metric spaces as inverse limits, and
bilipschitz embedding in $L_1$, {\it  Geom. Funct. Anal.}, {\bf
	23} (2013), 96--133.


\bibitem[DKO20]{DKO20} S.\,J. Dilworth, D.~Kutzarova,
M.\,I.~Ostrovskii, Lipschitz-free spaces on finite metric spaces,
{\it Canad. J. Math.},  {\bf 72} (2020), 774--804.

\bibitem[DKO21]{DKO21} S.\,J. Dilworth, D.~Kutzarova,
M.\,I.~Ostrovskii, Analysis on Laakso graphs with application to
the structure of transportation cost spaces. {\it Positivity} {\bf
	25} (2021), no. 4, 1403--1435.

\bibitem[FRT04]{FRT04} J.~Fakcharoenphol, S.~Rao,
	K.~Talwar, A tight bound on approximating arbitrary metrics by
	tree metrics, {\it J. Comput. System Sci.}, {\bf 69} (2004), no.
	3, 485--497.

\bibitem[GNS12]{GNS12} O.~Giladi, A.~Naor, G.~Schechtman, Bourgain's discretization
theorem, {\it Annales Mathematiques de la faculte des sciences de
	Toulouse}, vol. {\bf XXI} (2012), no. 4, 817--837.
	
\bibitem[GNRS04]{GNRS04} A.~Gupta, I.~Newman, Y.~Rabinovich,
	A.~Sinclair, Cuts, trees and $\ell_1$-embeddings of graphs, {\it
		Combinatorica}, {\bf 24} (2004) 233--269.

\bibitem[Hor07]{Hor07} K.\,J.~Horadam, 
{\it Hadamard matrices and their applications}. Princeton University Press, Princeton, NJ, 2007.

\bibitem[IW91]{IW91} M.~Imase, B.\,M.~Waxman, 
		Dynamic Steiner tree problem.
		{\it SIAM J. Discrete Math.} {\bf 4}  (1991), no. 3, 369--384.

	\bibitem[IT03]{IT03} P.~Indyk, N.~Thaper, 
\newblock Fast image retrieval via embeddings,
\newblock in: {\it ICCV 03: Proceedings of the 3rd International Workshop on Statistical and Computational Theories of Vision},  2003.

\bibitem[JWZ24]{JWZ24}
R.~Jayaram, E.~Waingarten, T.~Zhang. 
\newblock Data-dependent LSH for the earth mover’s distance. 
\newblock In {\it Proc. of the 56th Ann. ACM Symp. on Theory of Computing}, pages 800--811, 2024.

\bibitem[JS09]{JS09} W. B. Johnson, G. Schechtman, \emph{Diamond graphs and super-reflexivity,}  J. Topol. Anal., \textbf{
			1} (2009), no. 2, 177--189.
		
\bibitem[Kan42]{Kan42} L.\,V.~Kantorovich, On mass transportation
	(Russian), {\it Doklady Akad. Nauk SSSR}, (N.S.) {\bf 37}, (1942),
	199--201.
	
\bibitem[KG49]{KG49} L.\,V.~Kantorovich, M.\,K.~Gavurin, Application of mathematical methods in the analysis of cargo flows (Russian),
in: {\it Problems of improving of transport efficiency}, USSR
Academy of Sciences Publishers, Moscow, 1949, pp.
110--138.

\bibitem[KG81]{KG81} M.~Kaufman, R.\,B.~Griffiths
Exactly soluble Ising models on hierarchical lattices, 
Phys. Rev. B 24, 496-498

	\bibitem[KN06]{KN06} S. Khot, A.~Naor, Nonembeddability
	theorems via Fourier analysis, {\it Math. Ann.}, {\bf 334} (2006),
	821--852.

\bibitem[Kis75]{Kis75} S.~V.~Kislyakov,
Sobolev imbedding operators and the nonisomorphism of
certain Banach spaces, {\it Funct. Anal. Appl.}
{\bf 9} (1975), 290--294.


	\bibitem[KT02]{KT02} J.~Kleinberg,  \'E.~Tardos, Approximation algorithms for classification problems with pairwise relationships: metric labeling and Markov random fields. {\it J. ACM} {\bf 49} (2002), no. 5, 616–639.

\bibitem[Laa02]{Laa02} T. J. Laakso, \emph{Plane with $A_\infty$-weighted metric not bi-Lipschitz embeddable to $\mathbb{R}^N$,} 
Bull. London Math. Soc. \textbf{ 34} (2002), no. 6, 667–676.

\bibitem[LP01]{LP01} U. Lang, C. Plaut, \emph{Bilipschitz embeddings of metric
		spaces into space forms,}  Geom. Dedicata, \textbf{ 87}  (2001),  		285--307.
    
\bibitem[LN04]{LN04} J. R. Lee, A.~Naor, Embedding the diamond graph in $L_p$ and dimension reduction in $L_1$. {\it Geom. Funct. Anal.} {\bf 14} (2004), no. 4, 745–747. 
            
\bibitem[LR10]{LR10} J. R. Lee, P. Raghavendra, \emph{Coarse differentiation and
multi-flows in planar graphs,}  Discrete Comput. Geom. \textbf{  43} (2010), no. 2, 346--362.



\bibitem[MN11]{MN11} J. Matou\v sek (Editor), starting June 2010 maintained jointly with A.~Naor,  \emph{Open problems
		on embeddings of finite metric spaces,} last update August 2011,
		available from: \url{http://kam.mff.cuni.cz/~matousek/}.

    
	\bibitem[NPS20]{NPS20} A. Naor, G. Pisier,  G. Schechtman, Impossibility of dimension reduction in the nuclear norm. {\it Discrete Comput. Geom.} {\bf 63} (2020), no. 2, 319–345.
	
	
	\bibitem[NS07]{NS07} A. Naor, G. Schechtman, Planar Earthmover is not in
	$L_1$, {\it SIAM J. Computing}, {\bf 37} (2007), 804--826.

\bibitem[NY22]{NY22} A.~Naor, R.~Young, Foliated corona decompositions. {\it Acta Math.} {\bf 229} (2022), no. 1, 55–200. 
	
	\bibitem[NR03]{NR03}  I.~Newman, Y.~Rabinovich,
	Lower bound on the distortion of embedding planar metrics into
	Euclidean space, {\it Discrete Comput. Geom.} {\bf 29} (2003), no.
	1, 77--81.
	
\bibitem[Ost05]{Ost05} 
M.\,I.~Ostrovskii. 
\newblock Sobolev spaces on graphs.
\newblock {\it Quaestiones Mathematicae}, {\bf 28} (2005), 501--523.

		
\bibitem[Ost13]{Ost13} M. I. Ostrovskii, \emph{Metric Embeddings: Bilipschitz and Coarse Embeddings into Banach
			Spaces}, de Gruyter Studies in Mathematics, \textbf{ 49}. Walter de Gruyter \&\ Co., Berlin, 2013.

\bibitem[OR17]{OR17} M. I. Ostrovskii, B. Randrianantoanina, \emph{A new approach to low-distortion embeddings of finite metric spaces into non-superreflexive Banach spaces,}   J. Funct. Anal. \textbf{ 273} (2017), no. 2, 598--651.
				    
\bibitem[RTG98]{RTG98} 
Y.~Rubner, C.~Tomasi, L.\,J.~Guibas. 
\newblock A metric for distributions with applications to image databases.
\newblock In {\it Proc. of the 6th Int'l Conf. on Computer Vision}, pages 59--66, 1998.

\bibitem[RTG00]{RTG00} 
Y.~Rubner, C.~Tomasi, L.\,J.~Guibas. 
\newblock The earth mover’s distance as a metric for image retrieval.
\newblock{\it International journal of computer vision,} 40:99–121, 2000.

\bibitem[Ser03]{Ser03} Jean-Pierre Serre, {\it Trees}. Translated from the French original by John Stillwell. Corrected 2nd printing of the 1980 English translation. Springer Monographs in Mathematics. Springer-Verlag, Berlin, 2003.

\bibitem[Vil09]{Vil09}
	C.~Villani.
	\newblock {\it Optimal Transport: Old and New}.
	\newblock Grundlehren der mathematischen Wissenschaften, {\bf 338}, Springer, 2008.

\end{small}

\end{thebibliography}
\end{document}